\documentclass[12pt,english,twoside]{article}
\usepackage[T1]{fontenc}
\usepackage[latin9]{inputenc}
\usepackage{geometry}
\usepackage{mathrsfs}
\usepackage{amsthm}
\usepackage{amsmath}
\usepackage{amssymb}
\usepackage{esint}
\usepackage{color}

\definecolor{mhcol}{rgb}{0,0.7,0}
\definecolor{grey}{rgb}{0.5,0.5,0.5}

\newcommand{\MH}[1]{#1}
\newcommand{\MHrm}[1]{}

\makeatletter
\theoremstyle{plain}
\newtheorem{thm}{\protect\theoremname}
  \theoremstyle{plain}
  \newtheorem{assumption}[thm]{\protect\assumptionname}
  \theoremstyle{remark}
  \newtheorem{rem}[thm]{\protect\remarkname}
  \theoremstyle{plain}
  \newtheorem{lem}[thm]{\protect\lemmaname}
  \theoremstyle{definition}
  \newtheorem{defn}[thm]{\protect\definitionname}
  \theoremstyle{plain}
  \newtheorem{cor}[thm]{\protect\corollaryname}

\usepackage{fancyhdr}
\fancypagestyle{fancyfootonly}{%
  \fancyhf{}%
\lfoot{}
\cfoot{}
\rfoot{Berlin 2018}
  \renewcommand{\headrulewidth}{\z@}%
  \renewcommand{\footrulewidth}{\z@}%

}
\date{\vspace{-4.5ex}}

\makeatother

\usepackage{babel}
  \providecommand{\assumptionname}{Assumption}
  \providecommand{\corollaryname}{Corollary}
  \providecommand{\definitionname}{Definition}
  \providecommand{\lemmaname}{Lemma}
  \providecommand{\remarkname}{Remark}
\providecommand{\theoremname}{Theorem}

\usepackage[bookmarks]{hyperref}
\begin{document}
\global\long\def\womega{\tilde{\omega}}
\global\long\def\ue{u_{\eps}}

\global\long\def\fq{\mathfrak{q}}

\global\long\def\Q{\boldsymbol{Q}}

\global\long\def\weakto{\rightharpoonup}

\global\long\def\eps{\varepsilon}
\global\long\def\Z{\mathbb{Z}}
\global\long\def\Zd{\mathbb{Z}^{d}}
\global\long\def\B{\mathbb{B}}
\global\long\def\E{\mathbb{E}}
\global\long\def\N{\mathbb{N}}
\global\long\def\P{\mathbb{P}}
\global\long\def\R{\mathbb{R}}
\global\long\def\Rd{\mathbb{R}^{d}}
\global\long\def\Zde{\mathbb{Z}_{\eps}^{d}}
\global\long\def\Ztde{\mathbb{Z}_{\eps}^{2d}}
\global\long\def\fc{\mathfrak{c}}

\global\long\def\calL{\mathcal{L}}
\global\long\def\calR{\mathcal{R}}
\global\long\def\cC{\mathcal{C}}
\global\long\def\cE{\mathcal{E}}
\global\long\def\cQ{\mathcal{Q}}
\global\long\def\cR{\mathcal{R}}

\global\long\def\sE{\mathscr{E}}
\global\long\def\cH{\mathcal{H}}

\global\long\def\sF{\mathscr{F}}

\global\long\def\He{\mathcal{H}_{\eps}}
\global\long\def\d{\mathrm{d}}
\global\long\def\seps{\sum\!\!\!\!\!\!\,{\scriptstyle \eps}\,}
\global\long\def\sumeps#1{\underset{{\scriptstyle #1}}{\,\sum\!\!\!\!\!\!\,{\scriptstyle \eps}\;\;}}
\global\long\def\sumsum#1{\underset{{\scriptstyle #1}}{\sum\sum\;\;}}
\global\long\def\sumsumeps#1{\underset{{\scriptstyle #1}}{\sum\!\!\!\!\!\!\,{\scriptstyle \eps}\,\,\sum\!\!\!\!\!\!\,{\scriptstyle \eps}\;\;}}
\global\long\def\rmP{\mathrm{P}}

\global\long\def\PV{\mathrm{PV}}
\global\long\def\cB{\mathcal{B}}

\title{The fractional $p$-Laplacian emerging from homogenization of the random conductance model with degenerate ergodic weights and unbounded-range jumps}

\author{Franziska Flegel and Martin Heida}

\maketitle
\begin{abstract} We study a general class of discrete $p$-Laplace operators in the random
conductance model with long-range jumps and ergodic weights. Using a variational formulation 
of the problem, we show that under the assumption 
of bounded first moments and a suitable lower moment condition on the weights, the 
homogenized limit operator is a fractional $p$-Laplace operator.

Under strengthened lower moment conditions, we can apply our insights also to the spectral homogenization of the discrete
Laplace operator to the continuous fractional Laplace operator.

\end{abstract}
\thispagestyle{fancyfootonly}

\section{Introduction}

In a recent work \cite{FHS2018RCM}, the authors together with Slowik
studied homogenization of a discrete Laplace operator on $\Zde:=\eps\Zd$
with long range jumps of the form
\begin{equation}
\calL_{\eps}u(x):=\eps^{-2}\sum_{y\in\Zde\backslash\{x\}}{ \omega_{\frac{x}{\eps},\frac{y}{\eps}}}\left(u(y)-u(x)\right)\,.\label{eq:eps-laplace}
\end{equation}
The operator was studied on a bounded domain under proper rescaling with Dirichlet boundary
conditions. The coefficients $\omega_{x,y}$ being random and positive
with $\omega_{x,y}=\omega_{y,x}$, the operator $\calL_\eps$ acts on functions
$\Zde\to\R$, and the corresponding linear equation in \cite{FHS2018RCM} reads 
\begin{equation}
\calL_{\eps}u(x)=f(x)\,,\qquad u(x)=0\mbox{ on }\Zde\backslash\Q\,,\label{eq:linear-eps-problem}
\end{equation}
where $\Q$ is a bounded open domain in $\Rd$.
The assumptions on $\omega_{x,y}$ imposed in \cite{FHS2018RCM} are
ergodicity and stationarity in $x$, together with a first moment condition of the form 
\begin{equation}
\E\left(\sum_{z\in\Zd}\omega_{0,z}\left|z\right|^{2}\right)<\infty\,,\label{eq:moment-cond-FHS}
\end{equation}
and a lower moment condition of the form
\begin{equation}
\exists\fq>\frac{d}{2}\,:\quad\E\left(\sum_{i=1}^{d}\omega_{0,e_{i}}^{-\fq}+\omega_{0,-e_{i}}^{-\fq}\right)<\infty\,,\label{eq:moment-cond-FHS-2}
\end{equation}
where $e_{i}$ is the $i$-th unit vector in $\Zd$. Under these assumptions,
it could be shown that in the sense of G-convergence, the homogenized operator in the limit $\eps\to0$ 
is a second order elliptic operator on $L^{2}(\Q)$ 
of the form $\nabla\cdot\left(A_{\hom}\nabla\bullet\right)$ 
with Dirichlet boundary conditions, where $A_{\hom}\in\R^{d\times d}$ is symmetric and positive definite. 

\MH{The surprising result that the non local operator $\calL_\eps$ localizes in the limit $\eps\to0$ motivates us 
to explore this phenomenon in more detail, in particular to find out what assumptions on $\omega$ would
cause $\calL_\eps$ to remain non local in the limit $\eps\to 0$. 
By a non local limit operator, we mean a pseudo differential operator of fractional Laplacian type.
Noting that the second order type of the limit operator in \cite{FHS2018RCM} is strongly linked to the scaling $\eps^{-2}$ (refer also to \cite{Biskup2011review}), we first relax this scaling to $\eps^{-2s}$, $s\in(0,1)$ obtaining
\begin{align}
\calL_{\eps}u(x):&=\eps^{-2s}\sum_{y\in\Zde\backslash\{x\}}{ \omega_{\frac{x}{\eps},\frac{y}{\eps}}}\left(u(y)-u(x)\right) \nonumber \\
& =\eps^{d}\sum_{y\in\Zde\backslash\{x\}}c_{\frac{x}{\eps},\frac{y}{\eps}}\frac{u(y)-u(x)}{\left|x-y\right|^{d+2s}}\,,\label{eq:def-Leps-c}
\end{align}
where the new random variable $c$ on $\Z^{2d}$ relates to $\omega$
through 
\[
c_{x,y}:=\omega_{x,y}\left|x-y\right|^{d+2s}\,.
\]
Note that the prefactor $\eps^{d}$ balances $\left|x-y\right|^{-d}$ and $\eps^{-2s}$ is absorbed into $\left|x-y\right|^{-2s}$.} 

\MH{In case $c_{x,y}=c_0$ is constant for all $x,y$, it is intuitive that the limit operator of \eqref{eq:def-Leps-c} is no longer a second order elliptic operator but rather
a non local fractional operator of the form }
\[
\left(-\Delta\right)^{2s}u(x):=c_0\PV\int_{\Rd}\frac{u(y)-u(x)}{\left|x-y\right|^{d+2s}}\,\d y\,,
\]
where $\PV$ stands for \emph{principle value }of the integral. We
refer to \cite{kwasnicki2017ten} for a list of equivalent characterizations,
among which the most common is the Fourier-symbol $|\xi|^{2s}$. 

\MH{Our main theorems confirm our intuition in the case 
when $c$ is a positive random variable with finite, non-zero 
expectation $0<\E(c)<\infty$ and a suitable bound on the lower tail $0<\E(c^{-\fq})<\infty$ (see Assumption \ref{assu:fq} below).
Furthermore, we assume $c$ to be ergodic in $\Z^{d\times d}$. 
This is different from the ergodicity assumptions in \cite{FHS2018RCM}, 
where $\omega_{x,z}$ is ergodic only in the first variable. 
The reason is that $c_{x,y}=\omega_{x,y}|x-y|^{d+2}$ in \cite{FHS2018RCM} decreases to 0 with growing distance $|x-y|$, implying $\E(c)=0$ and  causing localization of the operator.
Hence no statistical independence w.r.t. the second parameter is needed in \cite{FHS2018RCM}. In contrast, in the present work we want to study non local limit behavior and in order to get a spatially homogeneous operator, we need some assumptions that provide good mixing conditions.}

\MH{Note that in view of \cite{FHS2018RCM} one could get the idea that our setting corresponds to a relaxation of condition (\ref{eq:moment-cond-FHS})
to, say 
\begin{equation}
\E\left(\sum_{z\in\Zd}\omega_{0,z}\left|z\right|^{2s}\right)<\infty\,,\qquad s\in(0,1)\,.\label{eq:s-moment-condition}
\end{equation}
However, our first moment condition is not equivalent with \eqref{eq:s-moment-condition} but corresponds to (see Lemma \ref{lem:E-omega-infty})
\[
\E\left(\sum_{z\in\Zd}\omega_{0,z}\left|z\right|^{2s}\right)=\infty\,.
\]
Moreover, Theorem \ref{thm:violate-infty-property} shows that \eqref{eq:s-moment-condition} causes the pseudo-differential operator to vanish in the limit $\eps\to0$.} 


In this work, we study the above homogenization problem in a more general setting. Our focus lies on energy functionals
which take the form 
\[
\sE_{p,s,\eps}(u)=\eps^{2d}\underset{{\scriptstyle (x,y)\in\Ztde}}{\sum\sum}c_{\frac{x}{\eps},\frac{y}{\eps}}\frac{V\left(u(x)-u(y)\right)}{\left|x-y\right|^{d+ps}}+\eps^{d}\sum_{x\in\Zde}G(u(x))-\eps^d\sum_{x\in\Zde}u(x)f_{\eps}(x)\,,
\]
where \MH{$V$ satisfies a lower $p$-growth condition (see Assumption \ref{assu:V} below) with $p\in(1,\infty)$ and $s\in(0,1)$, which is in accordance with the continuous theory of fractional $p$-Laplace operators.}
We will study both the convergence behavior on the whole of
$\Rd$ and on the restriction $u(x)=0$ for $x\not\in\Q$, where $\Q\subset\Rd$
is a bounded domain. 
\MH{Like in \cite{neukamm2017stochastic}, we study $\R^m$-valued $u$ but for simplicity, we will sometimes restrict in our discussion to $m=1$.} 
The corresponding limit functional (in the sense
of $\Gamma$-convergence) will turn out to be
\[
\sE_{p,s}(u)=\E(c)\iint_{\R^{2d}}\frac{V\left(u(x)-u(y)\right)}{\left|x-y\right|^{d+ps}}+\int_{\Rd}G(u(x))\d x-\int_{\Rd}u(x)f(x)\,.
\]

In case $V(\xi)=|\xi|^p$ this functional 
generates the fractional $p$-Laplace equation (see \cite{iannizzotto2016existence} and reference therein).
In what follows, we will shortly recall 
the relation between the homogenization problem for the linear equation and the 
homogenization of convex functionals.

In order to understand our way to approach this problem, note that
the weak formulation of (\ref{eq:linear-eps-problem}) with $\calL_{\eps}$
given by (\ref{eq:def-Leps-c}) reads 
\begin{equation}
\sum_{x\in\Zde}\eps^{d}\sum_{y\in\Zde}c_{\frac{x}{\eps},\frac{y}{\eps}}\frac{u(y)-u(x)}{\left|x-y\right|^{d+2s}}\left(v(y)-v(x)\right)=\sum_{x\in\Zde}f(x)v(x)\,.\label{eq:linear-eps-problem-2}
\end{equation}
\MH{We recall that the literature usually provides a factor $\frac12$ on the left hand side, which is not the case here as the sum in \eqref{eq:def-Leps-c} is over all neighbors and not only the neighbors in ''positive'' direction $e_i$.} 
In a variational formulation, $u$ is the minimizer of the energy
potential 
\[
\sE_{2,s,\eps}(u)=\eps^{2d}\frac{1}{2}\sum_{x\in\Zde}\sum_{y\in\Zde}c_{\frac{x}{\eps},\frac{y}{\eps}}\frac{\left(u(y)-u(x)\right)^{2}}{\left|x-y\right|^{d+2s}}-\eps^{d}\sum_{x\in\Zde}u(x)f(x)\,.
\]
We will also look at the constraint $u(x)=0$ on $\Zde\backslash\Q$.

In the continuum, a corresponding functional is known for the solutions
of the fractional Laplace equation $\left(-\Delta\right)^{s}u=f$
and on $\Q=\Rd$ it reads 
\[
\sE_{2,s}(u)=\frac12\iint_{\R^{2d}}\frac{\left(u(x)-u(y)\right)^{2}}{\left|x-y\right|^{d+2s}}-\int_{\Rd}u(x)f(x)\,.
\]
The minimizers of $\sE_{2,s}$ lie in the space $W^{s,p}(\Rd)$, which
we will introduce in Section \ref{sub:Discrete-and-continuous}. Hence,
a $\Gamma$-convergence result for $\sE_{2,s,\eps}\xrightarrow{\Gamma}\sE_{2,s}$
implies homogenization of (\ref{eq:linear-eps-problem-2}) to
\[
\left(-\Delta\right)^{s}u:=\E(c)\PV\int_{\Rd}\frac{u(y)-u(x)}{\left|x-y\right|^{d+2s}}\,\d y=f\,,
\]
see Section \ref{sub:Application-to-F-Lap}.

\MH{On bounded domains $\Q\subset\Rd$ we introduce the following functionals which 
are oriented at the definitions of $W^{s,p}(\Q)$-seminorms
in \cite{DiNezza2012}.} 
They read
\[
\sE_{p,s,\eps,\Q}(u)=\eps^{2d}\underset{{\scriptstyle (x,y)\in\Q^{\eps}\times\Q^{\eps}}}{\sum\sum}c_{\frac{x}{\eps},\frac{y}{\eps}}\frac{V\left(u(x)-u(y)\right)}{\left|x-y\right|^{d+ps}}+\eps^{d}\sum_{x\in\Q^{\eps}}G(u(x))-\eps^d\sum_{x\in\Q^{\eps}}u(x)f_{\eps}(x)\,,
\]
where $\Q^{\eps}=\Q\cap\Zde$. 
From the analytical point of view, it then makes sense to consider the
restriction of $\sE_{p,s,\eps,\Q}$ to functions with zero boundary
conditions and zero mean value conditions. In order to formulate discrete
Dirichlet conditions, let 
\begin{equation}
\partial\Q^{\eps}=\left\{ x\in\Zde\,:\;\partial\Q\cap\left(x+[-\eps,\eps]^{d}\right)\not=\emptyset\right\} \,.\label{eq:def-partial-Q-eps}
\end{equation}
In every of the above mentioned cases, the corresponding $\Gamma$-limit
functional will turn out to be
\[
\sE_{p,s,\Q}(u)=\E(c)\iint_{\Q\times\Q}\frac{V\left(u(x)-u(y)\right)}{\left|x-y\right|^{d+ps}}+\int_{\Q}G(u(x))\d x-\int_{\Q}u(x)f(x)\,.
\]
However, as we will see below, we even obtain a kind of Mosco convergence in suitable spaces $L^{r}(\Q)$.
Mosco convergence means that the $\liminf$-estimate can be obtained
for weakly converging sequences while the recovery sequence can be
constructed with respect to strong convergence. 

\MH{The homogenization on bounded $\Q$ announced above will be performed both for a constraint on the average value 
and for the constraint of Dirichlet boundary conditions. Here we have to be careful since 
the notion of boundary conditions in spaces $W^{s,p}(\Q)$ does
not make sense in case $s\leq\frac{1}{p}$. In this case, it is still possible to consider $\sE_{p,s,\eps}$ 
with the constraint that $u=0$ on $\Zde\setminus\Q$ and we therefore also study this particular situation.
}

Our convergence results rest upon a well-balanced interplay between
$p$, $s$, $c$ and $d$, which we formulate in the following condition on
the coefficients:

\begin{assumption}
\label{assu:fq} We assume that the random variable $c$ is ergodic
in $\Zd\times\Zd$ with $\E(c)<\infty$ and given $s\in(0,1)$, $p>1$
we assume that there exists $\fq\in\left(\frac{d}{ps},+\infty\right]$
and $r\in(1,p)$ such that $\E(c^{-\fq})<\infty$ and $\fq\geq\frac{r}{p-r}>\frac{d}{ps}$.
\end{assumption}

In the hypothetical case $s=1$ and $p=2$, the last assumption reduces
to $\fq>\frac{d}{2}$. Hence Assumption \ref{assu:fq} is in accordance
with the assumptions in \cite{FHS2018RCM}, which we recalled in 
(\ref{eq:moment-cond-FHS})--(\ref{eq:moment-cond-FHS-2}).

\begin{rem}
\label{rem:critical-compact-exponent} As we will see in Theorems
\ref{thm:Main Theorem-I}--\ref{thm:Main Theorem-V}, sequences $\ue$
with bounded $\sE_{p,s,\eps}(\ue)$ or $\sE_{\eps,p,\Q}(\ue)$ are
bounded in $L^{r}(\Zde)$ or $L^{r}(\Q^{\eps})$ if $r\in[1,p_{\fq}^{\star})$
for 
\[
p_{\fq}^{\star}=\frac{dp\fq}{2d+d\fq-sp\fq}\,.
\]
In particular, it turns out that $\fq>\frac{2d}{ps}$ is a sufficient
condition to have boundedness of $\ue$ in $L^{p}(\Q^{\eps})$. In
order to obtain suitable bounds on $\ue$ in $L^{r}(\Q)$, we ask
that $V$ satisfies the following assumption.\end{rem}
The notation $p_{\fq}^{\star}$ is related to the fractional critical exponent 
$p^{\star}$ in the classical theory of fractional Sobolev spaces, which is 
introduced in Theorem \ref{thm:discr-Poincare-wsp}. However, we will 
see that the random weights $c$ will force us to lower the value of the classical $p^{\star}$ with decreasing $\fq$.

We finally introduce our assumptions on $V$. These assumptions are a natural generalization of the fractional $p$-Laplace potential and are also natural in the context of Sobolev spaces which we will use.
\begin{assumption}
\label{assu:V}We assume that $V:\R^{m}\to\R$ is \MH{continuous} and there exist $\alpha,\beta,c>0$ and $p\in[1,\infty]$ such that 
\begin{gather*}
\alpha|\xi|^{p}\leq V(\xi)\leq c+\beta|\xi|^{p}\,,\\
\xi\mapsto|\xi|^{-p}V(\xi)\qquad\mbox{is continuous in }0\,.
\end{gather*}

\end{assumption}
The study of discrete elliptic operators has some history starting from works by K\"unne\-mann \cite{kunnemann1983diffusion} and Kozlov \cite{kozlov1987averaging}. The interest in this topic has been tremendous both from the physical point of view, e.g. as a model for Brownian motion (see \cite{Biskup2011review,Bouchaud1990}), or from mathematical point of view when studying numerical schemes (see \cite{Biskup2011review,kozlov1987averaging} or \cite{heida2018convergences} for a numerical application). The current research particularly focuses on higher order corrector estimates, see e.g. \cite{bella2017stochastic} and references therein. 

\MH{A further related work is by Neukamm, Sch\"affner and 
Schl\"omerkemper \cite{neukamm2017stochastic} on the homogenization of discrete 
non-convex functionals with finite range as discrete models for elasticity. Like in the present work, they allow for higher dimensional variables and non-convex $V$. In contrast to our approach, they do not weight the potential $V$ by $|x-y|^{-d-ps}$ and hence, even on bounded domains, the results of \cite{neukamm2017stochastic} and the present work cover different problems, though they fall in the same class of discrete homogenization. Since \cite{neukamm2017stochastic} treats only finite range interaction, the problem localizes and the homogenized potential is obtained from a sequence of ''cell-problems''. } 

\MH{We emphasize that all of the above mentioned works where on finite range connectivity. From the stochastic point of view of random walks among random conductances \cite{Biskup2011review}, this corresponds to Brownian motion of the random walker while our ansatz allows for long range jumps, which can be considered as discrete analogue of Levi-flights, such as are used to model the movement of bacteria.}

The homogenization of the fractional Laplace operator seems to be
only recent and rather unexplored. However, there are a few results
in the literature: Most of them are focused on the periodic homogenization
of the continuous fractional Laplace operator $(-\Delta)^{s}$, starting
from a work by Piatnitskii and Zhizhina \cite{piatnitskiZhiZhi2017periodic} \MH{and Kassmann, Piatnitskii and Zhizhina \cite{kassmann2018homogenization}}.
A first result on the stochastic homogenization of the (continuum)
fractional Laplace operator with uniformly bounded $c$ is given in
\cite{piatnitskiZhiZhi2018stochastic}. We will not investigate the
relation between \cite{piatnitskiZhiZhi2018stochastic} and the present
work, but we expect that the methods developed below could help to
generalize \cite{piatnitskiZhiZhi2018stochastic} to non-uniformly
bounded coefficients with bounded moment conditions. 

From the point of view of non local discrete operators, our work is 
related to our previous result \cite{FHS2018RCM} but also to a recent
result by Chen, Kumagai and Wang \cite{chen2018random}. \MH{They
show homogenization of the discrete fractional Laplace, i.e. $p=2$, on $\Zd$ in
case $d>4-4s$ and under the assumption $\E(c^{2p})+\E(c^{-q})<\infty$
where $p>\max\left\{\right(d+2)/d\,\,(d+1)/(2(2-2s))\}>1$ and $q>2\frac{d+2}{2}$. Note that the authors of \cite{chen2018random}
also allow for percolation with the restriction that $\P\left(c=0\right)<2^{-4}$, which we exclude for simplicity. Hence, some choices of $d$ and $s$ are contained both in the setting of \cite{chen2018random} and the present work, while there are other choices of $d$ and $s$ that are contained either in \cite{chen2018random} or in the present work but not in both. In this sense, the results are complementing each other.}

The outline of the paper is as follows: In the next section we first
provide Mosco convergence of $\sE_{p,s,\eps}$ and $\sE_{p,s,\eps,\Q}$
to $\sE_{p,s}$ and $\sE_{p,s,\Q}$ respectively. Recall that Mosco convergence
is slightly stronger than weak or strong $\Gamma$-convergence. Based
on these results, we formulate our homogenization results for the
fractional Laplace operator, including also spectral homogenization
in case $\fq>\frac{2d}{ps}$. In Section \ref{sec:Preliminaries}
we provide basic knowledge on fractional Sobolev spaces and generalize
these to the discrete setting. Lemma \ref{lem:weighted semi-poincare-general}
in Section \ref{sub:Weighted-discrete-Sobolev--Slobo} can be considered
as the heart of our homogenization results. Finally, in Section \ref{sec:Proof-of-Theorems}
we prove the main theorems from Section \ref{sec:Main-results}. For
readability of Section \ref{sec:Preliminaries}, we shift some standard
proofs to the appendix.

\section{\label{sec:Main-results}Main results }

The discrete space, on which our functionals $\sE_{p,s,\eps}$ and
$\sE_{p,s,\eps,\Q}$ are defined, are denoted 
\[
\cH_{\eps}:=\left\{ u:\,\Zde\to\R^m\right\} \,,\quad\mbox{resp.}\quad\cH_{\eps}(\Q):=\left\{ u\in\cH^{\eps}\,:\;\forall x\not\in\Q:\,u(x)=0\right\} \,.
\]
However, the limit functionals are defined on the measurable functions
on $\Rd$ and in order to compare discrete solutions with continuous
functions we introduce the operators $\cR_{\eps}^{\ast}$ through
\[
\cR_{\eps}^{\ast}u(x)=u(x_{i})\quad\mbox{if }x_{i}\in\Zde\mbox{ and }x\in x_{i}+\left[-\frac{\eps}{2},\frac{\eps}{2}\right)^{d}\,.
\]
As observed in \cite{FHS2018RCM}, the operator
$\cR_{\eps}^{\ast}$ is the dual of the operator 
\[
\left(\cR_{\eps}u\right)(x)=\eps^{-d}\int_{x_{i}+\left[-\frac{\eps}{2},\frac{\eps}{2}\right)^{d}}u(y)\d y\quad\mbox{if }x_{i}\in\Zde\mbox{ and }x\in x_{i}+\left[-\frac{\eps}{2},\frac{\eps}{2}\right)^{d}\,.
\]

\subsection{Homogenization of the global energy $\protect\sE_{p,s,\protect\eps}$ }

On bounded domains $\Q\subset\Rd$ we find the following convergence
behavior of $\sE_{p,s,\eps}$.
\begin{thm}
\label{thm:Main Theorem-I}Let $\Q\subset\Rd$ be a bounded domain. Let $c$, $s$, $p$, $\fq$ and $V$ satisfy
Assumptions \ref{assu:fq} and \ref{assu:V}, $G:\,\R^m\to\R$ be non-negative
and continuous with $G(\xi)\leq\alpha|\xi|^{k}$, $\alpha>0$, $k<p_{\fq}^{\star}$,
and let $f_{\eps}\in\cH_{\eps}$ be such that $\cR_{\eps}^{\ast}f_{\eps}\weakto f$
in $L^{r^{\ast}}(\Q)$, where $\frac{1}{r^{\ast}}+\frac{1}{p_{\fq}^{\star}}<1$.
Then the sequence $\sE_{p,s,\eps}$ restricted to $\He(\Q)$ Mosco-converges
\emph{almost surely} to $\sE_{p,s}$ in the following sense:
\begin{enumerate}
\item For $r=\frac{r^\ast}{r^\ast-1}$ there exists $C>0$ such that
\[
\forall\ue\in\He(\Q):\qquad\left\Vert \ue\right\Vert _{L^{r}(\Q^{\eps})}\leq C\sE_{p,s,\eps}(\ue) \quad\text{for all $\eps>0$}\,.
\]
For every sequence $\ue\in\He(\Q)$ such that $\sup_{\eps}\sE_{p,s,\eps}(\ue)<\infty$
there exists $u\in W^{s,p}(\Q)$, $u=0$ on $\Rd\backslash\Q$, and
a subsequence $\cR_{\eps}^{\ast}\ue\to u$ pointwise a.e. with $\cR_{\eps}^{\ast}\ue\to u$
strongly in $L^{r}(\Q)$, and 
\[
\liminf_{\eps\to0}\sE_{p,s,\eps}(\ue)\geq\sE_{p,s}(u)\,,
\]

\item For every $u\in W^{s,p}(\Q)$, $u=0$ on $\Rd\setminus\Q$ and $r=\frac{r^\ast}{r^\ast-1}$ there
exists a sequence $\ue\in\He(\Q)$ such that $\cR_{\eps}^{\ast}u^{\eps}\to u$
strongly in $L^{r}(\Q)$ and 
\begin{equation}
\limsup_{\eps\to0}\sE_{p,s,\eps}(\ue)=\sE_{p,s}(u)\,.\label{eq:Thm-Gamma-limsup}
\end{equation}

\end{enumerate}
\end{thm}
Note that for $\fq>\frac{2d}{ps}$ we can choose $r=p$. 
\begin{thm}
\label{thm:Main Theorem-II}Let $c$, $s$, $p$, $\fq$ and $V$
satisfy Assumptions \ref{assu:fq} and \ref{assu:V}, and let the
sequence $f_{\eps}$ and the function $G:\,\R^m\to\R$ satisfy either
one of the following conditions:
\begin{enumerate}
\item $G$ is non-negative and convex and there exists a bounded $C^{0,1}$
domain $\Q\subset\Rd$ such that every $f_{\eps}$ has support in
$\Q$. Furthermore $\cR_{\eps}^{\ast}f_{\eps}\weakto f$ in $L^{r^{\ast}}(\Q)$,
where $\frac{1}{r^{\ast}}+\frac{1}{p_{\fq}^{\star}}<1$.
\item $G(\xi)=\alpha|\xi|^{r}+\tilde{G}$, $\tilde{G}$ is non-negative
and convex and $r,r^{\ast}>1$ with $\frac{1}{r}+\frac{1}{r^{\ast}}=1$.
Furthermore, $\cR_{\eps}^{\ast}f_{\eps}\weakto f$ in $L^{r^{\ast}}(\Q)$. 
\end{enumerate}
Then the sequence $\sE_{p,s,\eps}$ restricted to $\He$ Mosco-converges
to $\sE_{p,s}$ in the following sense:
\begin{enumerate}
\item For every sequence $\ue\in\He$ such that $\sup_{\eps}\sE_{p,s,\eps}(\ue)<\infty$
there exists $u\in W^{s,p}(\Rd)$, and a subsequence $\eps'\to0$
such that $\cR_{\eps'}^{\ast}u_{\eps'}\to u$ pointwise almost everywhere
and 
\[
\liminf_{\eps\to0}\sE_{p,s,\eps}(\ue)\geq\sE_{p,s}(u)\,,
\]

\item For every $u\in W^{s,p}(\Rd)$ there exists a sequence $\ue\in\He$
such that $\cR_{\eps}^{\ast}u^{\eps}\to u$ pointwise almost everywhere
and 
\begin{equation}
\limsup_{\eps\to0}\sE_{p,s,\eps}(\ue)=\sE_{p,s}(u)\,.\label{eq:Thm-Gamma-limsup-Rd}
\end{equation}

\end{enumerate}
\end{thm}

\subsection{Homogenization of the local energy $\protect\sE_{p,s,\protect\eps,\protect\Q}$ }

The following two theorems deal with the homogenization of the functional
$\sE_{p,s,\eps,\Q}$. In this work, we will study $\sE_{p,s,\eps,\Q}$
with boundary conditions $\ue|_{\partial\Q^{\eps}}\equiv0$, mean
value conditions or with suitable conditions on $G$. In a first step,
we define the following spaces similar to the continuum case:
\begin{align*}
\cH_{\eps,0}(\Q) & :=\left\{ u\in\cH_{\eps}(\Q)\,:\;\forall x\in\partial\Q^{\eps}\;u(x)=0\right\} \,,\\
\cH_{\eps,(0)}(\Q) & :=\left\{ u\in\cH_{\eps}(\Q)\,:\;\sum_{x\in\Q^{\eps}}\;u(x)=0\right\} \,.
\end{align*}
\MH{As mentioned in the introduction, the consideration of bounded domains comes up with technical difficulties. These concern in particular uniform compact embeddings of $W^{s,p}(\Q^\eps)$ into $L^q(\Q^\eps)$. In $\Rd$, the embedding $W^{s,p}(\Q)$ into $L^q(\Q)$ is coupled to the property of $\Q$ being an extension domain. We replace this property by the concept of \emph{uniform extension domain}, see Definition \ref{def:uniform-extension-domain}, which is a domain that allows for a uniform bound for all $\eps>0$ on the extension operator of $W^{s,p}(\Q^\eps)$ into $W^{s,p}(\Zde)$. A class of domains that satisfy this property are rectangular domains. Similar to the continuous case the concept of uniform extension domains enables us to derive uniformly compact embedding results. Furthermore, in the case of Dirichlet boundary conditions, we can show that the canonical extension by $0$ is uniformly continuous for all Lipschitz domains. This motivates the formulation of the following two theorems.}

\begin{thm}
\label{thm:Main Theorem-III}Let $\Q\subset\Rd$ be a bounded $C^{0,1}$-domain. Let $c$, $s$, $p$, $\fq$ and $V$
satisfy Assumptions \ref{assu:fq} and \ref{assu:V}, $sp>1$, $G:\,\R^m\to\R$
non-negative and continuous with $G(\xi)\leq\alpha|\xi|^{m}$,$\alpha>0$,
$m<p_{\fq}^{\star}$, and let $f_{\eps}\in\cH_{\eps}$ be such that
$\cR_{\eps}^{\ast}f_{\eps}\weakto f$ in $L^{r^{\ast}}(\Q)$, where
$\frac{1}{r^{\ast}}+\frac{1}{p_{\fq}^{\star}}<1$. Then the sequence
$\sE_{p,s,\eps,\Q}$ restricted to $\cH_{\eps,0}(\Q)$ Mosco-converges
to $\sE_{p,s}$ restricted to $W_{0}^{s,p}(\Q)$ in the following
sense:
\begin{enumerate}
\item For every sequence $\ue\in\cH_{\eps,0}(\Q^\eps)$ such that $\sup_{\eps}\sE_{p,s,\eps}(\ue)<\infty$
there exists $u\in W_{0}^{s,p}(\Q)$, $u=0$ on $\Rd\backslash\Q$,
and a subsequence $\cR_{\eps}^{\ast}\ue\to u$ pointwise a.e. with
$\cR_{\eps}^{\ast}\ue\to u$ strongly in $L^{r}(\Q)$ for $r=\frac{r^\ast}{r^\ast-1}$
and 
\begin{equation}
\liminf_{\eps\to0}\sE_{p,s,\eps,\Q}(\ue)\geq\sE_{p,s}(u)\,,\label{eq:Thm-Gamma-limsup-Q-dirichlet-0}
\end{equation}

\item For every $u\in W_{0}^{s,p}(\Q)$ there exists a sequence $\ue\in\cH_{\eps,0}(\Q^\eps)$
such that $\cR_{\eps}^{\ast}u^{\eps}\to u$ strongly in $L^{r}(\Q)$
for $r=\frac{r^\ast}{r^\ast-1}$ and 
\begin{equation}
\limsup_{\eps\to0}\sE_{p,s,\eps,\Q}(\ue)=\sE_{p,s,\Q}(u)\,.\label{eq:Thm-Gamma-limsup-Q-dirichlet}
\end{equation}

\end{enumerate}
\end{thm}
If we do not consider zero Dirichlet boundary conditions, we have
to find a suitable replacement that guarantees that the necessary (compact)
embeddings hold. We use the concept of uniform extension domains introduced
in Definition \ref{def:uniform-extension-domain}. 
\begin{thm}
\label{thm:Main Theorem-V}Let $c$, $s$, $p$, $\fq$ and $V$ satisfy
Assumptions \ref{assu:fq} and \ref{assu:V}, $Q$ be a uniform extension domain in the sense of Definition \ref{def:uniform-extension-domain}, $G:\,\R^m\to\R$ non-negative
and continuous with $G(\xi)\leq\alpha|\xi|^{k}$,$\alpha>0$, $k<p_{\fq}^{\star}$,
and let $f_{\eps}\in\cH_{\eps}$ be such that $\cR_{\eps}^{\ast}f_{\eps}\weakto f$
in $L^{r^{\ast}}(\Q)$, where $\frac{1}{r^{\ast}}+\frac{1}{p_{\fq}^{\star}}<1$.
Then the sequence $\sE_{p,s,\eps,\Q}$ restricted to $\cH_{\eps,(0)}(\Q)$
Mosco-converges to $\sE_{p,s}$ restricted to $W_{(0)}^{s,p}(\Q)$
in the following sense:
\begin{enumerate}
\item For every sequence $\ue\in\cH_{\eps,(0)}(\Q^\eps)$ such that $\sup_{\eps}\sE_{p,s,\eps}(\ue)<\infty$
there exists $u\in W^{s,p}_{(0)}(\Q)$, $u=0$ on $\Rd\backslash\Q$, and
a subsequence $\cR_{\eps}^{\ast}\ue\to u$ pointwise a.e. with $\cR_{\eps}^{\ast}\ue\to u$
strongly in $L^{r}(\Q)$ for $r=\frac{r^\ast}{r^\ast-1}$ and \eqref{eq:Thm-Gamma-limsup-Q-dirichlet-0} holds.

\item For every $u\in W^{s,p}_{(0)}(\Q)$ there exists a sequence $\ue\in\cH_{\eps,(0)}(\Q^\eps)$
such that $\cR_{\eps}^{\ast}u_{\eps}\to u$ strongly in $L^{r}(\Q)$
for $r=\frac{r^\ast}{r^\ast-1}$ and \eqref{eq:Thm-Gamma-limsup-Q-dirichlet} holds.

\end{enumerate}
\end{thm}

We have already mentioned in the introduction that by Lemma \ref{lem:E-omega-infty}
\MH{
our above assumptions imply $\E\left(\sum_{z}\omega_{x,x+z}|z|^{ps}\right)=\infty$.
The following theorem shows that the assumption $\E\left(\sum_{z}\omega_{x,x+z}|z|^{ps}\right)<\infty$
would imply that the differential part of $\sE_{p,s,\eps,\Q}$ vanishes
in the limit of $\Gamma$-convergence. We formulate and prove the
result for $\sE_{p,s,\eps,\Q}$ with zero Dirichlet conditions but
note that the proof also works for zero mean value and for $\sE_{p,s,\eps}$.}
\begin{thm}
\label{thm:violate-infty-property}
\MH{Let $V(\xi)=|\xi|^{p}$ and let
$\E\left(\sum_{z}\omega_{x,x+z}|z|^{ps}\right)<\infty$, $G:\,\R\to\R$
non-negative and continuous with $G(\xi)\leq\alpha|\xi|^{m}$,$\alpha>0$,
$m<p_{\fq}^{\star}$, and let $f_{\eps}\in\cH_{\eps}$ be such that
$\cR_{\eps}^{\ast}f_{\eps}\weakto f$ in $L^{r^{\ast}}(\Q)$, where
$\frac{1}{r^{\ast}}+\frac{1}{p_{\fq}^{\star}}<1$. Then for every $u\in C_{0}^{1}(\Q)$
it holds 
\begin{equation}
\lim_{\eps\to0}\sE_{p,s,\eps,\Q}(u)=\int_{\Q}G(u(x))\d x-\int_{\Q}u(x)f(x)\,.\label{eq:thm:violate-infty-property}
\end{equation}
In particular, $\sE_{p,s,\eps,\Q}$ strongly $\Gamma$-converges in
$L^{m}(\Q)$ to 
\[
\tilde{\sE}_{p,s,\Q}(u)=\int_{\Q}G(u(x))\d x-\int_{\Q}u(x)f(x)
\]
in the sense that $\cR_{\eps}^{\star}\ue\to u$ in $L^{m}(\Q)$ implies
\begin{equation}
\liminf_{\eps\to0}\sE_{p,s,\eps,\Q}(\ue)\geq\tilde{\sE}_{p,s,\Q}(u)\label{eq:thm:violate-infty-property-liminf}
\end{equation}
and for $u\in L^{m}(\Q)$ there exists a sequence $\ue\in\He$ with
$\cR_{\eps}^{\star}\ue\to u$ in $L^{m}(\Q)$ and
\begin{equation}
\lim_{\eps\to0}\sE_{p,s,\eps,\Q}(\ue)=\tilde{\sE}_{p,s,\Q}(u)\,.\label{eq:thm:violate-infty-property-limsup}
\end{equation}
}\end{thm}

\begin{proof}
\MH{(\ref{eq:thm:violate-infty-property-liminf}) follows immediately by definition.
In order to prove (\ref{eq:thm:violate-infty-property-limsup}) we
use (\ref{eq:thm:violate-infty-property}) and observe that for
every sequence $u_{\delta}\in C_{c}^{1}(\Q)$ approximating $u\in L^{m}(\Omega)$ as $\delta\to0$
it holds 
\[
\left|\tilde{\sE}_{p,s,\Q}(u)-\sE_{p,s,\eps,\Q}(u_{\delta})\right|\leq\left|\tilde{\sE}_{p,s,\Q}(u)-\tilde{\sE}_{p,s,\Q}(u_{\delta})\right|+\left|\tilde{\sE}_{p,s,\Q}(u_{\delta})-\sE_{p,s,\eps,\Q}(u_{\delta})\right|\,.
\]
In this way, successively choosing first $u_{\delta}$ and then $\eps$
we construct the sequence $\ue$.}

\MH{It remains to prove (\ref{eq:thm:violate-infty-property}). By assumption,
it holds 
\[
\lim_{K\to\infty}\E\left(\sum_{\left|z\right|>K}\omega_{x,x+z}|z|^{ps}\right)=0\,.
\]
Defining $E:=\E\left(\sum_{z}\omega_{x,x+z}|z|^{ps}\right)$, we observe
for any $\delta>0$ and $\eps\ll\delta$
\begin{align*}
\eps^{-ps}\sumeps{x\in\Q^{\eps}}\sum_{y\in\Q^{\eps}}\omega_{\frac{x}{\eps},\frac{y}{\eps}}\left|u(x)-u(y)\right|^{p} & =\sumeps{x\in\Q^{\eps}}\sum_{y\in\Q^{\eps}}\omega_{\frac{x}{\eps},\frac{y}{\eps}}\left|\frac{x}{\eps}-\frac{y}{\eps}\right|^{ps}\frac{\left|u(x)-u(y)\right|^{p}}{\left|x-y\right|^{p}}\left|x-y\right|^{p-ps}\\
 & \le\left\Vert \nabla u\right\Vert _{\infty}^{p}\sumeps{x\in\Q^{\eps}}\sum_{\substack{y\in\Q^{\eps}\\
\left|x-y\right|<\delta
}
}\omega_{\frac{x}{\eps},\frac{y}{\eps}}\left|\frac{x}{\eps}-\frac{y}{\eps}\right|^{ps}\delta^{p-ps}+\\
 & \quad+\left\Vert \nabla u\right\Vert _{\infty}^{p}\sumeps{x\in\Q^{\eps}}\sum_{\left|z\right|\geq\frac{\delta}{\eps}}\omega_{\frac{x}{\eps},\frac{x}{\eps}+z}\left|z\right|^{ps}\left(\text{diam}\Q\right)^{p-ps}
\end{align*}
and hence in the limit it holds for all $K>0$:
\begin{multline*}
\lim_{\eps\to0}\eps^{-ps}\sumeps{x\in\Q^{\eps}}\sum_{y\in\Q^{\eps}}\omega_{\frac{x}{\eps},\frac{y}{\eps}}\left|u(x)-u(y)\right|^{p} \\
\leq\left\Vert \nabla u\right\Vert _{\infty}^{p}|\Q|\left(E\,\delta^{p-ps}+\left(\text{diam}\Q\right)^{p-ps}\E\left(\sum_{\left|z\right|>K}\omega_{x,x+z}|z|^{ps}\right)\right)\,.
\end{multline*}
This provides (\ref{eq:thm:violate-infty-property}).}
\end{proof}

\subsection{\label{sub:Application-to-F-Lap}Application to the (spectral) homogenization
of the fractional Laplace operator}

\MH{It is well-known that strong/weak $\Gamma$-convergence
of convex functionals implies strong/weak
convergence of the minimizers towards the minimizer of the limit functional, see \cite{DalMaso1993}.
Hence, we expect that solutions of $\calL_{\eps}\ue=f$ converge to solutions of
the fractional equation 
\[
\PV\int_{\Rd}\frac{u(y)-u(x)}{\left|x-y\right|^{d+2s}}\,\d y=f(x)\,.
\]
This indeed holds true and we recall the proof in the context of the following result.}
\begin{thm}\label{thm:spectral-thm-1}
Let the assumptions of Theorem \ref{thm:Main Theorem-I} hold with
$p=2$. For every $\eps>0$ there exists a unique solution $\ue\in\cH_{\eps}(\Q)$
such that for every $v\in\cH_{\eps}(\Q)$ it holds
\begin{equation}\label{eq:thm:spectral-thm-1-1}
\eps^{2d}\sumsum{x,y\in\Zde}c_{\frac{x}{\eps},\frac{y}{\eps}}\frac{\ue(y)-\ue(x)}{\left|x-y\right|^{d+2s}}\left(v(y)-v(x)\right)=\eps^{d}\sum_{x\in\Zde}f(x)v(x)\,,
\end{equation}
and as $\eps\to0$ we find $\cR_{\eps}^{\ast}\ue\to u$ strongly
in $L^{r}(\Q)$ and $u\in W^{s,2}(\Rd)$ is the unique solution to
the equation 
\begin{equation}\label{eq:thm:spectral-thm-1-2}
\forall v\in W^{s,2}(\Rd)\,:\quad\E(c)\int_{\Rd}\int_{\Rd}\frac{u(y)-u(x)}{\left|x-y\right|^{d+2s}}\left(v(y)-v(x)\right)\d x\d y=\int_{\Rd}fv\,,
\end{equation}
where $u=0$ outside of $\Q$. 
\end{thm}

\begin{proof}
Let $u$ be the unique minimizer of $\sE_{2,s}$ and let $\ue^{\ast}$
be a sequence such that $\cR_{\eps}^{\ast}\ue^{\ast}\to u$ strongly
in $L^{r}(\Q)$ and (\ref{eq:Thm-Gamma-limsup}) holds. Furthermore,
let $\ue\in\cH_{\eps}(\Q)$ be the minimizer of $\sE_{\eps,2}$ and
let $\tilde{u}=\lim_{\eps\to0}\cR_{\eps}^{\ast}\ue$ according to
Theorem \ref{thm:Main Theorem-I}. Then 
\[
\sE_{2,s}(u)=\lim_{\eps\to0}\sE_{2,s,\eps}(\ue^{\ast})\geq\liminf_{\eps\to0}\sE_{2,s,\eps}(\ue)\geq\sE_{2,s}(\tilde{u})\geq\sE_{2,s}(u)\,.
\]
where we used in the last inequality that $u$ is the minimizer of
$\sE_{2,s}$. Since the minimizer of $\sE_{2,s}$ is unique, we obtain
$\tilde{u}=u$ and the theorem is proved.
\end{proof}
In a similar way, we prove the following theorems.
\begin{thm}
Let the assumptions of Theorem \ref{thm:Main Theorem-II} hold with
$p=2$. For every $\eps>0$ there exists a unique solution $\ue\in\cH_{\eps}$
such that for every $v\in\cH_{\eps}$ equation \eqref{eq:thm:spectral-thm-1-1} holds
and as $\eps\to0$ we find $\cR_{\eps}^{\ast}\ue\to u$ pointwise
where $u\in W^{s,2}(\Rd)$ is the unique solution to \eqref{eq:thm:spectral-thm-1-2}.
\end{thm}

\begin{thm}
\label{thm:Main Theorem-III-2}Let the assumptions of Theorem \ref{thm:Main Theorem-III}
hold with $p=2$. For every $\eps>0$ there exists a unique solution
$\ue\in\cH_{\eps,0}(\Q)$ such that for every $v\in\cH_{\eps,0}(\Q)$
it holds
\begin{equation}
\eps^{2d}\sumsum{x,y\in\Q^{\eps}}c_{\frac{x}{\eps},\frac{y}{\eps}}\frac{\ue(y)-\ue(x)}{\left|x-y\right|^{d+2s}}\left(v(y)-v(x)\right)=\eps^{d}\sum_{x\in\Q^{\eps}}f(x)v(x)\,,\label{eq:thm:Main Theorem-III-2-1}
\end{equation}
and as $\eps\to0$ we find $\cR_{\eps}^{\ast}\ue\to u$ strongly
in $L^{r}(\Q)$ and $u\in W_{0}^{s,2}(\Q)$ is the unique solution
such that for every $v\in W_{0}^{s,2}(\Q)$ the following equation holds
\begin{equation}
\quad\E(c)\int_{\Q}\int_{\Q}\frac{u(y)-u(x)}{\left|x-y\right|^{d+2s}}\left(v(y)-v(x)\right)\d x\d y=\int_{\Q}fv\,.\label{eq:thm:Main Theorem-III-2-2}
\end{equation}

\end{thm}

\begin{thm}
Let the assumptions of Theorem \ref{thm:Main Theorem-V} hold with
$p=2$. For every $\eps>0$ there exists a unique solution $\ue\in\cH_{\eps,(0)}(\Q)$
such that for every $v\in\cH_{\eps,(0)}(\Q)$ equation \eqref{eq:thm:Main Theorem-III-2-1}
holds
and as $\eps\to0$ we find $\cR_{\eps}^{\ast}\ue\to u$ strongly
in $L^{r}(\Q)$ and $u\in W_{(0)}^{s,2}(\Q)$ is the unique solution
such that for every $v\in W_{(0)}^{s,2}(\Q)$ the equation \eqref{eq:thm:Main Theorem-III-2-2} holds.
\end{thm}

We finally take a look on the topic of spectral homogenization. Theorem
\ref{thm:discr-compact-weights} together with Remark \ref{rem:critical-compact-exponent}
and Theorem \ref{thm:Main Theorem-III} shows that the operators $\cB_{c}^{\eps}:\,\cH_{\eps}(\Q)\to\cH_{\eps,0}(\Q)$,
where $\cB_{c}^{\eps}(f)$ solves (\ref{eq:thm:Main Theorem-III-2-1}),
are uniformly compact with respect to the norm $L^{p}(\Q^{\eps})$.
Furthermore, Theorem \ref{thm:Main Theorem-III} yields that 
\[
\left\Vert \cR_{\eps}^{\ast}\cB_{c}^{\eps}f^{\eps}-u\right\Vert _{L^p(\Q)}\to0\quad\mbox{as }\eps\to0\,,
\]
if $\cR_{\eps}^{\ast}f^{\eps}\weakto f$ where $u$ is the solution
to (\ref{eq:thm:Main Theorem-III-2-2}). Furthermore, the solution
operator $\cB$ to (\ref{eq:thm:Main Theorem-III-2-2}) is compact
by the compact embedding $W^{s,2}(\Q)\hookrightarrow L^{2}(\Q)$.
Hence, we obtain the following result from \cite{JKO1994}, Theorem
11.4 and 11.5 following the argumentation in Section 8 of \cite{FHS2018RCM}.

\begin{thm}
Under the assumptions of Theorem \ref{thm:Main Theorem-III} let $\mu_k^\eps$ be the $k$-th eigenvalue (i.e. $\mu_1^\eps\geq\mu_2^\eps\geq\dots$) and $\psi_k^\eps$ the $k$-th eigenfunction of $\cB_c^\eps$. Furthermore, let $\mu_k$ be the $k$-th eigenvalue and $\psi_k$ the $k$-th eigenfunction of $\cB$. Then the following holds.
\begin{itemize}
	\item Let $k\in\N$ and let $\eps_m$ be a null sequence. Then there $\P$-a.s. exists a family $\{\psi_j^0\}_{1\leq j\leq k}$ of eigenvectors of $\cB$ and a subsequence still indexed by $\eps_m$  such that 
	$$\left(\cR_{\eps_m}^\ast\psi_1^{\eps_m},\dots,\cR_{\eps_m}^\ast\psi_k^{\eps_m}\right)\to \left(\psi_1^{0},\dots,\psi_k^{0}\right)\quad\text{strongly in }L^2(\Q)\,.$$
	\item If the multiplicity of $\mu_k$ is equal to $s$, i.e. 
		$$\mu_{k-1}>\mu_k=\mu_{k+1}=\dots=\mu_{k+s}>\mu_{k+s+1}$$
		then for $j=1,\dots,s$ there $\P$-a.s. exists a sequence $\psi^\eps\in\cH^\eps(\Q)$ such that
		$$\lim_{\eps\to0}\left\|\psi_{k+j}-\cR_\eps^\ast\psi^\eps\right\|_{L^2(\Q)}=0$$
		where $\psi^\eps$ is a linear combination of the eigenfunctions of the operator $\cB_c^\eps$ corresponding to $\mu_k^\eps,\dots,\mu_{k+s}^\eps$.
\end{itemize}
\end{thm}

\section{Preliminaries\label{sec:Preliminaries}}

We first fix some convenient notation for discrete integrals (i.e.
higher dimensional sums) and function spaces. For $A\subset\Rd$ we
write $\left|A\right|_{\eps}:=\eps^{d}\sharp\left\{ A\cap\Zde\right\} $
and note that $\left|A\right|_{\eps}\to\left|A\right|$ as $\eps\to\infty$
for every open set $A\subset\Rd$. Moreover, for $A\subset\Zde$ and
a function $f:\,A\to\R$ we define 
\[
\sumeps{x\in A}\,f(x):=\eps^{d}\sum_{x\in A}f(x)\,.
\]
Then, for every function $f\in C_{c}(\Rd)$ we find 
\[
\sumeps{x\in\Zde}f(x)\to\int_{\Rd}f\,.
\]
Hence, ${\displaystyle \sumeps{}}$ is a discrete equivalent of the
integral $\int$.

\subsection{\label{sub:Discrete-and-continuous}Discrete and continuous Sobolev--Slobodeckij
spaces}

We introduce the Sobolev--Slobodeckij space $W^{s,p}(\Rd)$ as the
closure of $C_{c}^{\infty}(\Rd)$ with respect to the norm 
\[
\left\Vert u\right\Vert _{s,p}^{p}:=\left\Vert u\right\Vert _{L^{p}(\Rd)}^{p}+\left[u\right]_{s,p}^{p}\,,\quad\mbox{where}\quad\left[u\right]_{s,p}^{p}:=\int_{\Rd}\int_{\Rd}\frac{\left|u(x)-u(y)\right|^{p}}{\left|x-y\right|^{d+sp}}\,\d x\,\d y
\]
is the Gagliardo seminorm. This family of spaces is discussed in detail
for example in \cite{DiNezza2012,Triebel1978}. In general, they can
be constructed as the interpolation of $W^{1,p}(\Rd)$ and $L^{p}(\Rd)$,
see e.g. \cite{adams1975sobolev,Triebel1978}, but in this work, we
follow the outline of \cite{DiNezza2012}. We also consider Sobolev--Slobodeckij
spaces $W^{s,p}(\Q)$ on Lipschitz bounded domains $\Q\subset\Rd$.
These are defined by the norm $\left\Vert u\right\Vert _{s,p,\Q}^{p}:=\left\Vert u\right\Vert _{L^{p}(\Q)}^{p}+\left[u\right]_{s,p,\Q}^{p}$,
where the semi-norm $\left[u\right]_{s,p,\Q}^{p}$ is given through
\[
\left[u\right]_{s,p,\Q}^{p}=\int_{\Q}\int_{\Q}\frac{\left|u(x)-u(y)\right|^{p}}{\left|x-y\right|^{d+sp}}\,\d x\,\d y\,.
\]
As can be found for example in Theorem 5.4 of \cite{DiNezza2012},
\begin{equation}
\mbox{the extension operator }W^{s,p}(\Q)\hookrightarrow W^{s,p}(\Rd)\mbox{ is continuous for every }s\in(0,1]\label{eq:Wsp-cont-extent}
\end{equation}
if $\partial\Q$ is bounded and of class $C^{0,1}$. Property (\ref{eq:Wsp-cont-extent})
is called the $W^{s,p}$-extension property of domains $\Q$ and it
is used to prove compactness of embeddings $W^{s,p}(\Q)\hookrightarrow W^{s',p}(\Q)$
for $0<s'<s<1$ and $W^{s,p}(\Q)\hookrightarrow L^{q}(\Q)$ for every
1>$s>0$ and $\frac{s}{d}+\frac{1}{q}-\frac{1}{p}>0$. If $\partial\Q$
is bounded and of class $C^{0,1}$ and $sp>1$, it makes sense to
consider 
\[
W_{0}^{s,p}(\Q):=\left\{ u\in W^{s,p}(\Q)\,:\;u|_{\partial\Q}\equiv0\right\} \,,
\]
as in this case the trace is well defined. 
\begin{rem}\label{rem:Sobolev-interpolate}
In general, the space
\begin{equation}\label{eq:def-sobolev-diri-frac}
W_{0}^{s,p}(\Q):=\left(L^{p}(\Q),W_{0}^{1,p}(\Q)\right)_s 
\end{equation} 
is the interpolate of $W_{0}^{1,p}(\Q)$
and $L^{p}(\Q)$ and hence the extension by $0$ 
to $W_{0}^{s,p}(\Q)\hookrightarrow W^{s,p}(\Q)$ is continuous and well defined (see \cite[VII.7.17]{adams1975sobolev}). Interestingly, \eqref{eq:def-sobolev-diri-frac}
is well defined also in case $sp\leq 1$ but on the whole $W^{s,p}(\Q)\hookrightarrow W^{s,p}(\Rd)$.
Heuristically, this stems from the fact that $sp\leq 1$ implies that functions might have jumps
across Lipschitz manifolds. Thus, we may formally identify $W_0^{s,p}(\Q)= W^{s,p}(\Q)$ for $sp\leq 1$.
\end{rem}

A further space we will
use is 
\[
W_{(0)}^{s,p}(\Q):=\left\{ u\in W^{s,p}(\Q)\,:\;\int_{\Q}u=0\right\} \,.
\]
On $\Rd$ we do not have compact embedding but it holds that $W^{s,p}(\Rd)\hookrightarrow L^{q}(\Rd)$
continuously for every $q\in[p,p^{\star}]$, where $p^{\star}=dp/(d-sp)$
for $sp<d$. Furthermore, the set $C_{c}^{\infty}(\Rd)$ is dense
in $W^{s,p}(\Rd)$. We finally need the following approximation result.
\begin{lem}
\label{lem:Dirac-Approx}Let $\eta\in C_{c}^{\infty}(B_{1}(0))$ with
$\eta\geq0$ and $\int\eta=1$ and for $k\in\N$ denote $\eta_{k}(x):=\eta(kx)$.
Denoting $f\ast\eta_{k}$ the convolution of a measurable function
$f$ and $\eta_{k}$ we find for every $u\in W^{s,p}(\Rd)$ that 
\[
\left\Vert u\ast\eta_{k}\right\Vert _{s,p}\leq\left\Vert u\right\Vert _{s,p}\qquad\mbox{and}\qquad\lim_{k\to\infty}\left\Vert u\ast\eta_{k}-u\right\Vert _{s,p}=0\,.
\]

\end{lem}
We shift the proof to the appendix, as it is standard.

In this work, we will need a discrete notion of Sobolev--Slobodeckij
spaces and generalizations of the above embedding results to the discrete
setting. To this aim, we consider the following normed subspaces of
$\cH_{\eps}$. First, set $\Q^{\eps}:=\Zde\cap\Q$ for a bounded domain
$\Q\subset\Rd$ to define 
\[
\left\Vert u\right\Vert _{L^{p}(\Zde)}^{p}:=\sumeps{x\in\Zde}|u(x)|^{p}\quad\mbox{and}\quad\left\Vert u\right\Vert _{L^{p}(\Q^{\eps})}^{p}:=\sumeps{x\in\Q^{\eps}}|u(x)|^{p}\,,
\]
and let $W^{s,p}(\Zde)$ be the closure of $C_{c}^{\infty}(\Rd)$
with respect to the norm 
\[
\left\Vert u\right\Vert _{s,p,\eps}^{p}:=\left\Vert u\right\Vert _{L^{p}(\Zde)}^{p}+\left[u\right]_{s,p,\eps}^{p}\,,\quad\mbox{where}\quad\left[u\right]_{s,p,\eps}^{p}:=\sumeps{x\in\Zde}\sumeps{y\in\Zde}\frac{\left|u(x)-u(y)\right|^{p}}{\left|x-y\right|^{d+sp}}\,.
\]
When restricted to a bounded domain $\Q\subset\Rd$, we define $\left\Vert u\right\Vert _{s,p,\eps,\Q}^{p}:=\left\Vert u\right\Vert _{L^{p}(\Q^{\eps})}^{p}+\left[u\right]_{s,p,\eps,\Q}^{p}$
the norm of the space $W^{s,p}(\Q^{\eps})$, where 
\begin{equation}
\left[u\right]_{s,p,\eps,\Q}^{p}:=\sumeps{x\in\Q^{\eps}}\sumeps{y\in\Q^{\eps}}\frac{\left|u(x)-u(y)\right|^{p}}{\left|x-y\right|^{d+sp}}\,.\label{eq:sob-slob-semi-norm-Q-eps}
\end{equation}
For some of the proofs below, we need a discrete version of the continuous
extension property (\ref{eq:Wsp-cont-extent}) which holds uniformly
in $\eps$. As announced in the introduction we formulate this condition
in a definition.
\begin{defn}
\label{def:uniform-extension-domain}A bounded domain $\Q\subset\Rd$
is called a \emph{uniform extension domain} if there exists $C>0$
such that for every $\eps>0$ there exists a linear extension operator
$\cE_{\eps}\,:\,\,W^{s,p}(\Q^{\eps})\hookrightarrow W^{s,p}(\Zde)$
with $\left\Vert \cE_{\eps}\right\Vert \leq C$.\end{defn}
\begin{rem}
\label{rem:Extention-bounded-support}We may assume for a uniform
extension domain $\Q$ that there exists a further bounded domain
$\tilde{\Q}\supset\overline{\Q}$ and such that the extensions have
compact support in $\tilde{\Q}$. We prove this in the appendix.
\end{rem}
We will not go into details on this point but note that being a uniform extension domain
is immediate for rectangular boxes $\Q=\prod_{i=1}^{d}(a_{i},b_{i})$,
where $-\infty<a_{i}<b_{i}<+\infty$ for every $i=1,\dots d$. This
can be checked by reflection at the boundaries. Furthermore, Theorem \ref{thm:embedding-dirichlet}
suggests that every $C^{0,1}$ domain should be a uniform extension domain. However, the proof
of such a statement is beyond the scope of this work.

In the following, we formulate the four most important results of
this subsection. The proofs are technical and either standard ( and
hence shifted to the appendix ) or will be presented in Section \ref{sub:Proof-poincare}
below.
\begin{thm}[Discrete Sobolev inequality on $\Zde$]
\label{thm:discr-Poincare-wsp}Let $s\in(0,1)$ and $p\in[1,\infty)$
be such that $sp<d$ and let $p^{\star}:=dp/(d-sp)$. Then, for every
$q\in[p,p^{\star}]$, there exists a constant $C_{p,q}>0$ depending
only on $d$, $p$, $q$ and $s$ such that for every $\eps>0$ and
every $u\in W^{s,p}(\Zde)$ it holds
\begin{equation}
\left\Vert u\right\Vert _{L^{q}(\Zde)}\leq C_{p,q}\left\Vert u\right\Vert _{s,p,\eps}\,.\label{eq:discr-poinc-zde}
\end{equation}

\end{thm}
The exponent $p^{\star}$ is called the \emph{fractional critical exponent}.
As a corollary, the last result extends to bounded domains.
\begin{thm}
\label{thm:discr-Poincare-wsp-Qeps}Let $\Q\subset\Rd$ be a bounded
uniform extension domain and let $s\in(0,1)$ and $p\in[1,\infty)$
be such that $sp<d$ and let $p^{\star}:=dp/(d-sp)$. Then, for every
$q\in[p,p^{\star}]$, there exists a constant $C_{p,q}>0$ depending
only on $d$, $p$, $q$, $s$ and $\Q$ such that for every $\eps>0$
and every $u\in W^{s,p}(\Q^{\eps})$ it holds
\[
\left\Vert u\right\Vert _{L^{q}(\Q\cap\Zde)}\leq C_{p,q}\left\Vert u\right\Vert _{s,p,\eps,\Q}\,.
\]

\end{thm}
Furthermore, we obtain the following compactness result on bounded
domains.
\begin{thm}
\label{thm:Compactness-discrete-eps-to-0}Let $\Q\subset\Rd$ be a
bounded uniform extension domain and let $s\in(0,1)$ and $p\in[1,\infty)$.
Let $p^{\star}:=dp/(d-sp)$ if $sp<d$, and $p^{\star}=\infty$ else.
For every $\eps>0$ let $\ue\in W^{1,p}(\Q^{\eps})$ such that $\sup_{\eps>0}\left\Vert \ue\right\Vert _{s,p,\eps,\Q}<\infty$.
Then, for every $q\in[p,p^{\star})$ the family $\left(\cR_{\eps}^{\ast}\ue\right)_{\eps>0}$
is precompact in $L^{q}(\Q^{\eps})$.
\end{thm}
The proofs of Theorems \ref{thm:discr-Poincare-wsp} and \ref{thm:Compactness-discrete-eps-to-0}
are very technical and mostly follow the outline of proofs from \cite{DiNezza2012}.
Hence, for better readability of the paper, we shift them to the appendix.

Finally, we turn to Poincaré-type inequalities on bounded domains
with Dirichlet boundary conditions or zero mean value. We hence define
the spaces 
\begin{align*}
W_{0}^{s,p}(\Q^{\eps}) & :=\left\{ u\in W^{s,p}(\Q^{\eps})\,:\;u|_{\partial\Q^{\eps}}\equiv0\right\} \,,\\
W_{(0)}^{s,p}(\Q^{\eps}) & :=\left\{ u\in W^{s,p}(\Q^{\eps})\,:\;\sumeps{x\in\Q^{\eps}}u=0\right\} \,.
\end{align*}
 The corresponding embedding theorems are the following.
\begin{thm}
\label{thm:embedding-dirichlet}Let $\Q\subset\Rd$ be a bounded domain
with $C^{0,1}$ boundary, let $p\in(1,\infty)$, $s\in(0,1)$. Identifying
every function $u\in W_{0}^{s,p}(\Q^{\eps})$ with its extension by
$0$ outside $\Q^{\eps}$, there exists $C>0$ independent from $\eps$
such that 
\begin{equation}
\forall u\in W_{0}^{s,p}(\Q^{\eps})\,:\quad\left[u\right]_{s,p,\eps}\leq C\left[u\right]_{s,p,\eps,\Q}\,.\label{eq:PI-WspepsQ-Dir-1}
\end{equation}
For every $q\in[p,p^{\star}]$, there exists a constant $C_{p,q}>0$
depending only on $d$, $p$, $q$, $s$ and $\Q$ such that for every
$\eps>0$ it holds
\begin{equation}
\forall u\in W_{0}^{s,p}(\Q^{\eps})\,:\quad\left\Vert u\right\Vert _{L^{q}(\Q\cap\Zde)}\leq C_{p,q}\left[u\right]_{s,p,\eps,\Q}^{p}\label{eq:PI-WspepsQ-Dir-2}
\end{equation}
Finally, let $p^{\star}:=dp/(d-sp)$ if $sp<d$, and $p^{\star}=\infty$
else. For every $\eps>0$ let $\ue\in W^{1,p}(\Q^{\eps})$ such that
$\sup_{\eps>0}\left\lfloor \ue\right\rfloor _{s,p,\eps,\Q}<\infty$.
Then, for every $q\in[p,p^{\star})$ the family $\left(\cR_{\eps}^{\ast}\ue\right)_{\eps>0}$
is precompact in $L^{q}(\Q^{\eps})$.
\end{thm}
Furthermore, we have a similar result in case $W_{0}^{s,p}(\Q^{\eps})$
is replaced by $W_{(0)}^{s,p}(\Q^{\eps})$.
\begin{thm}
\label{thm:embedding-mean}Let $\Q\subset\Rd$ be a bounded uniform
extension domain with $C^{0,1}$ boundary, let $p\in(1,\infty)$,
$s\in(0,1)$. For every $q\in[p,p^{\star}]$, there exists a constant
$C_{p,q}>0$ depending only on $d$, $p$, $q$, $s$ and $\Q$ such
that for every $\eps>0$ it holds
\begin{equation}
\forall u\in W_{(0)}^{s,p}(\Q^{\eps})\,:\quad\left\Vert u\right\Vert _{L^{q}(\Q\cap\Zde)}\leq C_{p,q}\left[u\right]_{s,p,\eps,\Q}^{p}\label{eq:PI-WspepsQ-mean-2}
\end{equation}
Finally, let $p^{\star}:=dp/(d-sp)$ if $sp<d$, and $p^{\star}=\infty$
else. For every $\eps>0$ let $\ue\in W^{1,p}(\Q^{\eps})$ such that
$\sup_{\eps>0}\left\lfloor \ue\right\rfloor _{s,p,\eps,\Q}<\infty$.
Then, for every $q\in[p,p^{\star})$ the family $\left(\cR_{\eps}^{\ast}\ue\right)_{\eps>0}$
is precompact in $L^{q}(\Q^{\eps})$.
\end{thm}
The proof of Theorems \ref{thm:embedding-dirichlet} and \ref{thm:embedding-mean}
is given in the following subsection. It will be based on the fact
that $W^{s,p}(\Zde)$ embeds into $W^{s,p}(\Rd)$ via a finite element
interpolation operator.

\subsection{\label{sub:Proof-poincare}Proof of Theorems \ref{thm:embedding-dirichlet}
and \ref{thm:embedding-mean} }

We first study an interesting connection between $W^{s,p}(\Rd)$ and
$W^{s,p}(\Zde)$. Let 
\[
\rmP:\,\,[0,1]\times\{0,1\}\to\R\,,\quad(x,\kappa)\mapsto\begin{cases}
x & \mbox{if }\kappa=1\\
1-x & \mbox{if }\kappa=0
\end{cases}\,,
\]
we define for $x=\left(x_{j}\right)_{j=1\dots d}$ and $\kappa=\left(\kappa_{j}\right)_{j=1\dots d}\in\{0,1\}^{d}$
and $\varphi\in\cH_{\eps}$: 
\[
\left(\cQ_{\eps}\varphi\right)(x):=\sum_{\kappa\in\{0,1\}^{d}}\varphi\left(\eps\left\lfloor \frac{x}{\eps}\right\rfloor +\eps\kappa\right)\prod_{j=1}^{d}\rmP\left(\left\{ \frac{x_{j}}{\eps}\right\} ,\kappa_{j}\right)\,,
\]
the finite element interpolation of $\varphi$. Our first corollary
on the operator $\cQ_{\eps}$ is the following. 
\begin{cor}
Let $p\in[1,\infty)$. There exists a constant $C>0$ for every $\varphi\in\cH_{\eps}$
\begin{equation}
C^{-1}\left\Vert \varphi\right\Vert _{L^{p}(\Zde)}\leq\left\Vert \cQ_{\eps}\varphi\right\Vert _{L^{p}(\Zde)}\leq C\left\Vert \varphi\right\Vert _{L^{p}(\Zde)}\,.\label{eq:lem:Qphi-phi-lp-equivalence}
\end{equation}
 
\end{cor}
This corollary is straight forward to prove from the definition of
$\cQ_{\eps}$. Moreover, we obtain the following natural property.
\begin{lem}
\label{lem:Qphi-phi-wsp-equivalence}Let $p\in[1,\infty)$ and $s\in(0,1)$.
Then there exists $C>0$ such that for every $\eps>0$
\begin{equation}
\forall\varphi\in\cH_{\eps}\,:\quad\left[\cQ_{\eps}\varphi\right]_{s,p}^{p}\leq C\left[\varphi\right]_{s,p,\eps}^{p}\label{eq:lem:Qphi-phi-wsp-equivalence}
\end{equation}
\end{lem}
\begin{proof}
For $\kappa\in\{0,1\}^{d}$ we write $\kappa^{i,0}$
and $\kappa^{i,1}$ for the vectors where the $i$-th entry of $\kappa$
is replaced by $0$ and $1$ respectively. In order to reduce notation,
we write 
\[
\left(\delta_{i}^{\eps}\varphi\right)(x,\kappa):=\varphi\left(\eps\left\lfloor \frac{x}{\eps}\right\rfloor +\eps\kappa^{i,1}\right)-\varphi\left(\eps\left\lfloor \frac{x}{\eps}\right\rfloor +\eps\kappa^{i,0}\right)
\]
and hence obtain
\begin{equation}
\partial_{i}\cQ_{\eps}\varphi=\frac{1}{\eps}\sum_{\kappa\in\{0,1\}^{d}}\frac{1}{2}\left(\delta_{i}^{\eps}\varphi\right)(x,\kappa)\prod_{j\not=i}\rmP\left(\left\{ \frac{x_{j}}{\eps}\right\} ,\kappa_{j}\right)\,.\label{eq:lem:Qphi-phi-wsp-equivalence-help-1}
\end{equation}
For every $x\in\Rd$ let $\left\lfloor x\right\rfloor _{\eps}\in\Zde$
be the unique element such that $x\in\cC_{\eps}(x):=\left\lfloor x\right\rfloor _{\eps}+[0,\eps)^{d}$.
We denote $\overline{x}_{\eps}$ the center of $\cC_{\eps}(x)$ and
define 
$$A_{\eps}(x):=\Zde\cap\left(\eps\left\lfloor \frac{x}{\eps}\right\rfloor +[-\eps,\eps]^{d}\right)$$
as well as $B_{\eps}(x)=\overline{x}_{\eps}+[-\frac{3}{2}\eps,\frac{3}{2}\eps]^{d}$ and
$B_\eps^\complement(x):=\Rd\setminus B_\eps(x)$.
We then find for $\tilde{\varphi}=\cQ_{\eps}\varphi$ 
\begin{equation}
\left[\tilde\varphi \right]_{s,p}^{p}\leq\sum_{z\in\Zde}\int_{z+(0,\eps)^{d}}\d x\left(\int_{B_{\eps}(x)}\frac{\left|\tilde{\varphi}(x)-\tilde{\varphi}(y)\right|^{p}}{\left|x-y\right|^{d+sp}}\d y+\int_{B_{\eps}^{\complement}(x)}\frac{\left|\tilde{\varphi}(x)-\tilde{\varphi}(y)\right|^{p}}{\left|x-y\right|^{d+sp}}\d y\right)\,.\label{eq:lem:Qphi-phi-wsp-equivalence-help-2}
\end{equation}
Now, observe that with (\ref{eq:lem:Qphi-phi-wsp-equivalence-help-1})
it holds 
\begin{align*}
\int_{B_{\eps}(x)}\frac{\left|\tilde{\varphi}(x)-\tilde{\varphi}(y)\right|^{p}}{\left|x-y\right|^{d+sp}}\d y & \leq\left\Vert \nabla\tilde{\varphi}\right\Vert _{C^{\infty}(B_{\eps}(x))}^{p}\int_{B_{\eps}(x)}\left|x-y\right|^{p-sp-d}\d y\\
 & \leq C\left\Vert \nabla\tilde{\varphi}\right\Vert _{C^{\infty}(B_{\eps}(x))}^{p}\eps^{p-sp}\\
 & \leq C\eps^{-sp}\sum_{z\in A_{\eps}(x)}\sum_{i=1}^d\sum_{\tilde{\kappa}\in\{0,1\}^{d}}\eps^d\frac{\left|\left(\delta_{i}^{\eps}\varphi\right)(z,\kappa)\right|^{p}}{\eps^d}\\
 & \leq C\eps^{d}\sum_{y,z\in B_{\eps}(x)}\frac{\left|\varphi(z)-\varphi(y)\right|^{p}}{\left|z-y\right|^{d+ps}}\,,
\end{align*}
where $C$ changes in each line but is independent from $\eps$ and
$\varphi$. Furthermore, estimating $\frac{\left|\tilde{\varphi}(x)-\tilde{\varphi}(y)\right|^{p}}{\left|x-y\right|^{d+sp}}$
over each cell $\cC_{\eps}(y)$ it is easy to verify (see also the proof of Lemma \ref{lem:discr-isoperimetric-ineq} in the appendix) that we have
\[
\int_{B_{\eps}^{\complement}(x)}\frac{\left|\tilde{\varphi}(x)-\tilde{\varphi}(y)\right|^{p}}{\left|x-y\right|^{d+sp}}\d y\leq C\sum_{z\in\Zde\cap\left(\overline{x}_{\eps}+(-\eps,\eps)^{d}\right)}\sum_{y\in\Zde\backslash\left(\overline{x}_{\eps}+(-\eps,\eps)^{d}\right)}\eps^{d}\frac{\left|\varphi(z)-\varphi(y)\right|^{p}}{\left|z-y\right|^{d+sp}}\,.
\]
Hence the term in brackets on the right hand side of (\ref{eq:lem:Qphi-phi-wsp-equivalence-help-2})
is independent from $x\in z+(0,\eps)^{d}$ and we find 
\[
\left[u\right]_{s,p}^{p}\leq C\sumeps{x\in\Zde}\left(\sumeps{y\in\Zde\backslash\left(\overline{x}_{\eps}+(-\eps,\eps)^{d}\right)}\frac{\left|\varphi(x)-\varphi(y)\right|^{p}}{\left|x-y\right|^{d+sp}}+\sumeps{y\in A_{\eps}(x)}\frac{\left|\varphi(x)-\varphi(y)\right|^{p}}{\left|x-y\right|^{d+ps}}\right)\,.
\]
Since $C$ does not depend on $\eps$ or $\varphi$, this finally
yields (\ref{eq:lem:Qphi-phi-wsp-equivalence}).
\end{proof}

\begin{proof}[Proof of Theorem \ref{thm:embedding-dirichlet}]
Let $u\in W_{0}^{s,p}(\Q^{\eps})$. Due to (\ref{eq:lem:Qphi-phi-lp-equivalence})
and (\ref{eq:lem:Qphi-phi-wsp-equivalence}) we know that $\cQ_{\eps}u\in W_{0}^{s,p}(\Q)$.
We can now extend $v^{\eps}:=\cQ_{\eps}u$ to $\Rd$ by $0$ and obtain
$v^{\eps}\in W^{s,p}(\Rd)$ with $\left\Vert v^{\eps}\right\Vert _{s,p}\leq C\left\Vert v^{\eps}\right\Vert _{s,p,\Q}$,
where $C>0$ depends on $s$, $p$ and $\Q$. This follows from Remark \ref{rem:Sobolev-interpolate}.

We now show $\left\Vert u\right\Vert _{s,p,\eps}\leq C\left\Vert v^{\eps}\right\Vert _{s,p}$.
Since $\left\Vert v^{\eps}\right\Vert _{s,p}\leq C\left\Vert u\right\Vert _{s,p,\eps,\Q}$ by Lemma \ref{lem:Qphi-phi-wsp-equivalence},
this in turn implies the theorem by virtue of Theorems \ref{thm:discr-Poincare-wsp}
and \ref{thm:Compactness-discrete-eps-to-0}.

It only remains to show that 
\begin{equation}
\sumeps{x\in\Zde\backslash\Q}\,\sumeps{y\in\Q^{\eps}\backslash\partial\Q^{\eps}}\frac{\left|u(y)\right|^{p}}{\left|x-y\right|^{d+ps}}<C\left\Vert v^{\eps}\right\Vert _{s,p}\,.\label{eq:Poinc-in-diri-help-1}
\end{equation}
For this reason, let $x\in\Zde\backslash\Q$ and $y\in\Q^{\eps}\backslash\partial\Q^{\eps}$.
Then by definition of $\partial\Q^{\eps}$ in (\ref{eq:def-partial-Q-eps})
it holds $\left|x-y\right|\geq2\eps$. Let $\tilde{x}\in x+\left[-\frac{\eps}{2},\frac{\eps}{2}\right]^{d}$,
$\tilde{y}\in y+\left[-\frac{\eps}{2},\frac{\eps}{2}\right]^{d}$. In order to provide an upper bound on $|\tilde x - \tilde y|$ in terms of $|x-y|$ assume that $x\in y+[-2\eps,2\eps]$. It then
holds $\eps\leq\left|\tilde{x}-\tilde{y}\right|\leq3d^{\frac{1}{d}}\eps$.
Hence we can conclude that $3d^{\frac{1}{d}}\left|x-y\right|\geq\left|\tilde{x}-\tilde{y}\right|$. In case $x\not\in y+[-2\eps,2\eps]$ the ratio between $\left|x-y\right|$ and $\left|\tilde{x}-\tilde{y}\right|$ becomes smaller.
Furthermore, since $u\geq0$, we have $\cQ_{\eps}u(\tilde{y})\geq2^{-d}u(y)$
and 
\begin{multline*}
\int_{x+\left[-\frac{\eps}{2},\frac{\eps}{2}\right]^{d}}\int_{y+\left[-\frac{\eps}{2},\frac{\eps}{2}\right]^{d}}\frac{\left|\cQ_{\eps}u(\tilde{y})\right|^{p}}{\left|\tilde{x}-\tilde{y}\right|^{d+ps}}\d\tilde{y}\,\d\tilde{x}\\
\geq2^{-d}\left(3d^{\frac{1}{d}}\right)^{-d-ps}\int_{x+\left[-\frac{\eps}{2},\frac{\eps}{2}\right]^{d}}\int_{y+\left[-\frac{\eps}{2},\frac{\eps}{2}\right]^{d}}\frac{\left|u(y)\right|^{p}}{\left|x-y\right|^{d+ps}}\d\tilde{y}\,\d\tilde{x}\,.
\end{multline*}
Summing up the last inequality over $x$ and $y$ yields (\ref{eq:Poinc-in-diri-help-1}). 
\end{proof}

\begin{proof}[Proof of Theorem \ref{thm:embedding-mean}]
 Let us first verify that (\ref{eq:PI-WspepsQ-mean-2}) holds. Assume
that (\ref{eq:PI-WspepsQ-mean-2}) was wrong. Without loss of generality,
we might assume that $q>p$. In particular, we use $\left\Vert u_{\eps_{k}}\right\Vert _{L^{q}(\Q^{\eps_{k}})}\leq C\left\Vert u_{\eps_{k}}\right\Vert _{L^{p}(\Q^{\eps_{k}})}$.
Then there exists a sequence $\left(\eps_{k}\right)_{k\in\N}$, $\eps_{k}>0$,
and a sequence of functions $u_{\eps_{k}}\in\cH_{\eps_{k},(0)}(\Q)$
such that 
\begin{align*}
\left\Vert u_{\eps_{k}}\right\Vert _{L^{q}(\Q^{\eps_{k}})}=1 & \geq k\,\left[u_{\eps_{k}}\right]_{s,p,\eps_{k},\Q}^{p}\,,
\end{align*}
and we find $\cR_{\eps_{k}}^{\ast}u_{\eps_{k}}\to u$ strongly in
$L^{q}(\Q)$ by Theorem \ref{thm:Compactness-discrete-eps-to-0}.
But then $u=0$ since $\cR_{\eps_{k}}^{\ast}u_{\eps_{k}}\weakto0$
weakly in $L^{q}(\Q)$. This is a contradiction. The compactness
follows from Theorem \ref{thm:Compactness-discrete-eps-to-0}.
\end{proof}

\subsection{Dynamical systems}

Throughout this paper, we follow the setting of Papanicolaou and Varadhan
\cite{papanicolaou1979boundary} and make the following assumptions.
\begin{assumption}
\label{assu:Omega-mu-tau}Let $D\in\N$ and let $(\Omega,\sF,\P)$
be a probability space with a given family $(\tau_{x})_{x\in\Z^{D}}$
of measurable bijective mappings $\tau_{x}:\Omega\mapsto\Omega$,
having the properties of a \emph{dynamical system }on $(\Omega,\sF,\P)$,
i.e. they satisfy (i)-(iii):
\begin{enumerate}
\item [(i)]$\tau_{x}\circ\tau_{y}=\tau_{x+y}$ , $\tau_{0}=id$ (Group
property)
\item [(ii)]$\P(\tau_{-x}B)=\P(B)\quad\forall x\in\Z^{D},\,\,B\in\sF$
(Measure preserving)
\item [(iii)]$A:\,\,\Zd\times\Omega\rightarrow\Omega\qquad(x,\omega)\mapsto\tau_{x}\omega$
is measurable (Measurability of evaluation) 
\end{enumerate}
Let the system $(\tau_{x})_{x\in\Z^{D}}$ be ergodic i.e. for every
$\sF$-measurable set $B\subset\Omega$ holds
\begin{equation}
\left[\P\left((\tau_{x}(B)\cup B)\backslash(\tau_{x}(B)\cap B)\right)=0\,\,\forall x\in\Zd\right]\Rightarrow[\P(B)\in\{0,1\}]\,.\label{eq:def_ergodicity-2}
\end{equation}
\end{assumption}
\begin{thm}[Ergodic Theorem \cite{Daley1988} Theorem 10.2.II and also \cite{tempel1972ergodic}]
\label{thm:Ergodic-Theorem}Let $\left(A_{n}\right)_{n\in\N}$ be
a family of convex sets in $\Z^{D}$ such that $A_{n+1}\subset A_{n}$
and such that there exists a sequence $r_{n}$ with $r_{n}\to\infty$
as $n\to\infty$ such that $B_{r_{n}}(0)\cap\Z^{D}\subseteq A_{n}$.
If $\left(\omega_{x}\right)_{x\in\Z^{D}}$ is a stationary ergodic
random variable with finite expectation, then almost surely 
\begin{equation}
\frac{1}{\#A_{n}}\sum_{x\in A_{n}}\omega_{x}\to\E(\omega)\,.\label{eq:ergodic convergence}
\end{equation}

\end{thm}
The last theorem has an important consequence for our work:
\begin{lem}
\label{lem:ergodic-bound}Let $\left(A_{n}\right)_{n\in\N}$ be a
family of convex sets in $\R^{D}$ such that $A_{n+1}\subset A_{n}$
and such that there exists a sequence $r_{n}$ with $r_{n}\to\infty$
as $n\to\infty$ such that $B_{r_{n}}(0)\subseteq A_{n}$. If $\left(c_{x}\right)_{x\in\Z^{D}}$
is a stationary ergodic random variable with finite expectation, then
almost surely
\begin{equation}
\fc:=\sup_{\eps,n}\frac{\eps^{D}}{\left|A_{n}\right|}\sum_{x\in A_{n}\cap\Z_{\eps}^{D}}c_{\frac{x}{\eps}}<\infty\,.\label{eq:ass-cR-exists}
\end{equation}
\end{lem}
\begin{proof}
Defining $\fc_{n,\eps}:=\frac{\eps^{D}}{\left|A_{n}\right|}\sum_{x\in A_{n}\cap\Z_{\eps}^{D}}c_{\frac{x}{\eps}}$
we observe that Theorem \ref{thm:Ergodic-Theorem} implies 
\begin{equation}
\forall\eps>0\,:\,\,\fc_{n,\eps}\to\E(\tilde{c})\mbox{ as }n\to\infty\,,\qquad\mbox{and}\qquad\forall n\,:\,\,\fc_{n,\eps}\to\E(c)\mbox{ as }\eps\to0\,.\label{eq:pointw-c-R-k-eps}
\end{equation}
Assume that (\ref{eq:ass-cR-exists}) was wrong. Then there exists
a sequence $\left(n_{k},\eps_{k}\right)_{n\in\N}$ such that $\fc_{n_{k},\eps_{k}}\to\infty$
as $k\to\infty$. If we assume $n_{k}$ was bounded by $N$, then
the second part of (\ref{eq:pointw-c-R-k-eps}) implies existence
of $C>0$ such that 
\[
\sup_{\eps_{k},n_{k}}\fc_{n_{k},\eps_{k}}\leq\sup_{n\leq N}\sup_{\eps}\fc_{n,\eps}<C<\infty\,,
\]
which is a contradiction to the assumption that (\ref{eq:ass-cR-exists})
was wrong. Hence we can w.l.o.g. assume $n_{k}\uparrow\infty$. 

By the same argument, we can assume $\eps_{k}\downarrow0$. But then,
the Ergodic Theorem \ref{thm:FHS-Ergodic-Thm} implies $\fc_{n_{k},\eps_{k}}\to\E(\tilde{c})$,
a contradiction with $\fc_{n_{k},\eps_{k}}\to\infty$. Hence (\ref{eq:ass-cR-exists})
holds.
\end{proof}
A further important consequence is the following:
\begin{lem}
\label{lem:ergodic-bound-polynomial}Let $c$ be a random variable
on $\Zd\times\Zd$ satisfying Assumption \ref{assu:fq}. Then for
every bounded convex domain $\Q\subset\Rd$ and every $\alpha,\xi>0$
it holds 
\[
\sup_{\eps>0}\sumeps{x\in\Q^{\eps}}\sumeps{\substack{\left|x-y\right|<\xi\\
y\in\Zde
}
}c_{\frac{x}{\eps},\frac{y}{\eps}}\left|x-y\right|^{-d+\alpha}<C\xi^{\alpha}\,,
\]
where $C$ only depends on $\Q$ and $d$.\end{lem}
\begin{proof}
We consider 
\begin{align*}
\frac{1}{|\Q|}\sumeps{x\in\Q^{\eps}}\sumeps{\substack{\left|x-y\right|<\xi\\
y\in\Zde
}
}c_{\frac{x}{\eps},\frac{y}{\eps}}\left|x-y\right|^{-d+\alpha} & =\sum_{k=0}^{\infty}\frac{1}{|\Q|}\sumeps{\substack{x\in\Q\\
x\in\Zde
}
}\sumeps{\substack{\frac{1}{2}\xi\leq2^{k}\left|x-y\right|<\xi\\
y\in\Zde
}
}c_{\frac{x}{\eps},\frac{y}{\eps}}\left|x-y\right|^{-d+\alpha}\\
 & \leq\sum_{k=0}^{\infty}\left(2^{-k}\xi\right)^{\alpha}\frac{1}{|\Q|}\sumeps{\substack{x\in\Q\\
x\in\Zde
}
}\left(2^{-k-1}\xi\right)^{-d}\sumeps{\substack{\frac{1}{2}\xi\leq2^{k}\left|x-y\right|<\xi\\
y\in\Zde
}
}c_{\frac{x}{\eps},\frac{y}{\eps}}\\
 & =\sum_{k=0}^{\infty}\left(2^{-k}\xi\right)^{\alpha}\frac{1}{|\Q|}\sumeps{\substack{x\in2^{k}\Q\\
x\in\Z_{2^{k}\eps}^{d}
}
}\left(2^{-k-1}\xi\right)^{-d}\sumeps{\substack{\frac{1}{2}\xi\leq\left|x-y\right|<\xi\\
y\in\Z_{2^{k}\eps}^{d}
}
}c_{\frac{x}{2^{k}\eps},\frac{y}{2^{k}\eps}}\\
 & \leq\left(2\xi\right)^{-d}\sum_{k=0}^{\infty}\left(2^{-k}\xi\right)^{\alpha}\frac{\left(2^{k}\eps\right)^{2d}}{2^{kd}|\Q|}\sum_{\substack{x\in2^{k}\Q\\
x\in\Z_{2^{k}\eps}^{d}
}
}\sum_{\substack{\left|x-y\right|<\xi\\
y\in\Z_{2^{k}\eps}^{d}
}
}c_{\frac{x}{2^{k}\eps},\frac{y}{2^{k}\eps}}\,.
\end{align*}
Replacing $\delta:=2^{k}\eps$, $\Q_{k}:=2^{k}\Q$, $\xi_{k}=(1+k)\xi$
and $\tilde{\Q}_{k}:=\left\{ (x,y)\,:\;x\in\Q_{k},\,|x-y|<\xi_{k}\right\} $
we obtain 
\[
\frac{\eps^{2d}}{|\Q|}\sum_{x\in\Q^{\eps}}\sum_{\substack{\left|x-y\right|<\xi\\
y\in\Zde
}
}c_{\frac{x}{\eps},\frac{y}{\eps}}\left|x-y\right|^{-d+\alpha}\leq2^{-d}\sum_{k=0}^{\infty}\left(2^{-k}\xi\right)^{\alpha}\frac{\delta^{2d}}{|\tilde{\Q}_{k}|}\left(1+k\right)^{d}\sum_{\substack{\left(x,y\right)\in\tilde{\Q}_{k}\\
\left(x,y\right)\in\Z_{\delta}^{d}\times\Z_{\delta}^{d}
}
}c_{\frac{x}{2^{k}\eps},\frac{y}{2^{k}\eps}}\,.
\]
By Lemma \ref{lem:ergodic-bound} the sequence 
\[
\fc:=\sup_{\delta,k}\frac{\delta^{2d}}{|\tilde{\Q}_{k}|}\sum_{\substack{\left(x,y\right)\in\tilde{\Q}_{k}\\
\left(x,y\right)\in\Z_{\delta}^{d}\times\Z_{\delta}^{d}
}
}c_{\frac{x}{2^{k}\eps},\frac{y}{2^{k}\eps}}<\infty
\]
is bounded. Hence we observe 
\[
\frac{1}{|\Q|}\sumeps{x\in\Q^{\eps}}\sumeps{\substack{\left|x-y\right|<\xi\\
y\in\Zde
}
}c_{\frac{x}{\eps},\frac{y}{\eps}}\left|x-y\right|^{-d+\alpha}\leq\xi^{\alpha}2^{-d}\fc\sum_{k=0}^{\infty}2^{-k\alpha}\xi_{k}^{d}\,,
\]
and since $\sum_{k=0}^{\infty}2^{-k\alpha}\xi_{k}^{d}$ is bounded,
the lemma is proved.
\end{proof}
We will need to test the convergence (\ref{eq:ergodic convergence})
with a pointwise converging sequence of functions. The following necessary
result by Flegel, Heida and Slowik is a generalization of \cite[Theorem 3]{BoivinDepauw2003}.
\begin{thm}[{Extended ergodic Theorem \cite[Theorem 5.2]{FHS2018RCM}}]
\label{thm:FHS-Ergodic-Thm}Let $\Q\subset\Z^{D}$ be a convex set
containing $0$ and let $f$ be a stationary random ergodic variable
on $\Z^{D}$ with finite expectation. Furthermore, let $u_{\eps}:\Z_{\eps}^{D}\to\R$
be a sequence of functions such that $\cR_{\eps}^{\ast}\ue\to u$ pointwise
a.e. and $\sup_{\eps}\left\Vert \ue\right\Vert _{\infty}<\infty$.
Then
\[
\eps^{D}\sum_{x\in\Q\cap\Z_{\eps}^{D}}f(\frac{x}{\eps})\ue(x)\to\E(f)\int_{\Q}u(x)\d x\qquad\mbox{as }\eps\to0\,.
\]

\end{thm}
As a direct consequence of the above ergodic theorems we obtain the
following result on our coefficients $\omega$ and $c$.
\begin{lem}
\label{lem:E-omega-infty}Let $0<\E(c)<\infty$ and let $c_{x,y}=\omega_{x,y-x}\left|x-y\right|^{d+ps}$.
Then 
\[
\E\left(\sum_{z\in\Zd}\omega_{0,z}\left|z\right|^{ps}\right)=\infty\,.
\]
\end{lem}
\begin{proof}
For every $R>0$ and every $k<R$ we have 
\begin{align*}
\frac{1}{R^{d}}\sum_{\substack{x\in\Zd\\
|x|\leq R/2
}
}\sum_{\substack{z\in\Zd\\ |z|<R}}\omega_{x,z}|z|^{ps} 
 & =\frac{1}{R^{d}}\sum_{\substack{x\in\Zd\\
|x|\leq R/2
}
}\sum_{\substack{y\in\Zd\\
|y-x|\leq R
}
}c_{x,y}\frac{1}{R^{d}}+\frac{1}{R^{d}}\sum_{\substack{x\in\Zd\\
|x|\leq R/2
}
}\sum_{\substack{z\in\Zd\\
|z|\leq R
}
}\left(1-\frac{|z|^{d}}{R^{d}}\right)\omega_{x,z}|z|^{ps}\\
 & >\frac{1}{R^{2d}}\sum_{\substack{x\in\Zd\\
|x|\leq R/2
}
}\sum_{\substack{y\in\Zd\\
|y|\leq R/2
}
}c_{x,y}+\frac{1}{R^{d}}\sum_{\substack{x\in\Zd\\
|x|\leq R/2
}
}\sum_{\substack{z\in\Zd\\
|z|\leq R\,k^{-1/d}
}
}\frac{k-1}{k}\omega_{x,z}|z|^{ps}\,.
\end{align*}
Hence, passing to the limit $R\to\infty$ on both sides we obtain
\[
|S^{d-1}|\E\left(\sum_{z\in\Zd}\omega_{0,z}\left|z\right|^{ps}\right)\geq|S^{d-1}|^{2}\E\left(c\right)+\frac{k-1}{k}|S^{d-1}|\E\left(\sum_{z\in\Zd}\omega_{0,z}\left|z\right|^{ps}\right)\,,
\]
where $|S^{d-1}|$ is the surface of the $d$-dimensional unit ball in $\Rd$. Since the last inequality holds for arbitrary $k\in\N$ we find $|S^{d-1}|\E\left(\sum_{z\in\Zd}\omega_{0,z}\left|z\right|^{ps}\right)=\infty$.

\end{proof}

\subsection{\label{sub:Weighted-discrete-Sobolev--Slobo}Weighted discrete Sobolev--Slobodeckij
spaces}

This section is concerned with the (compact) embedding of discrete
weighted Sobolev--Slobodeckij spaces into the discrete Sobolev--Slobodeckij
spaces from Section \ref{sub:Discrete-and-continuous}. More precisely,
the heart of this section (and of the whole article) is the inequality 
\begin{equation}
\left(\sumeps{x\in\Q^{\eps}}\sumeps{y\in\Q^{\eps}}\frac{\left|u(x)-u(y)\right|^{r}}{\left|x-y\right|^{d+rs'}}\right)^{\frac{1}{r}}\leq C\left(\sumeps{x\in\Q^{\eps}}\sumeps{y\in\Q^{\eps}}c_{\frac{x}{\eps},\frac{y}{\eps}}\frac{\left|u(x)-u(y)\right|^{p}}{\left|x-y\right|^{d+ps}}\right)^{\frac{1}{p}}\label{eq:lem:weighted semi-poincare-general}
\end{equation}
for suitable $r>1$ and $s'\in(0,s)$, where $C$ should depend on
$s$, $s'$, $p$ and $r$ but not on $\eps$. Let us first establish
some conditions on $c_{x,y}$, $s\in(0,1)$ and $p\in(1,\infty]$,
under which we can expect existence of suitable $r$ and $s'$. For
simplicity of notation, we establish the following semi-norm corresponding
to (\ref{eq:sob-slob-semi-norm-Q-eps}):

\[
\left[u\right]_{s,p,\eps,\Q,c}:=\left(\sumeps{x\in\Q^{\eps}}\sumeps{y\in\Q^{\eps}}c_{\frac{x}{\eps},\frac{y}{\eps}}\frac{\left|u(x)-u(y)\right|^{p}}{\left|x-y\right|^{d+ps}}\right)^{\frac{1}{p}}\,.
\]
We can use Hölder's inequality and observe that 
\begin{align}
\sumeps{x\in\Q^{\eps}}\sumeps{y\in\Q^{\eps}}\frac{\left|u(x)-u(y)\right|^{r}}{\left|x-y\right|^{d+rs'}} & =\sumeps{x\in\Q^{\eps}}\sumeps{y\in\Q^{\eps}}\frac{\left(c_{\frac{x}{\eps},\frac{y}{\eps}}\right)^{\frac{r}{p}}}{\left(c_{\frac{x}{\eps},\frac{y}{\eps}}\right)^{\frac{r}{p}}}\frac{\left|u(x)-u(y)\right|^{r}}{\left|x-y\right|^{\left(d+ps\right)\frac{r}{p}}}\left|x-y\right|^{d\frac{r}{p}-d+r(s-s')}\nonumber \\
 & \leq\left[u\right]_{s,p,\eps,\Q,c}^{r}\left(\sumeps{x\in\Q^{\eps}}\sumeps{y\in\Q^{\eps}}\left(c_{\frac{x}{\eps},\frac{y}{\eps}}\right)^{-\frac{r}{p-r}}\left|x-y\right|^{-d\left(1-\frac{rp}{d\left(p-r\right)}(s-s')\right)}\right)^{\frac{p-r}{p}}\,.\label{eq:lem:wsp-help-1}
\end{align}
In order to obtain (\ref{eq:lem:weighted semi-poincare-general}),
it is necessary to show that the second factor on the right hand side
of (\ref{eq:lem:wsp-help-1}) is uniformly bounded in $\eps>0$.
We have to distinguish two cases.

In the first case, we assume that $1-\frac{rp}{d\left(p-r\right)}(s-s')\leq0$,
which is equivalent to
\begin{align}
\frac{r}{p-r}\geq\frac{d}{p(s-s')}
\label{eq:lem:wsp-case-1}
\end{align}
and can be fulfilled for a suitable $s'\in (0,s)$ if and only if $\frac{r}{p-r}>\frac{d}{ps}$.
In this case, the factor $\left|x-y\right|^{-d\left(1-\frac{rp}{d\left(p-r\right)}(s-s')\right)}$ stays bounded since $\Q^{\eps}$ is bounded.
It follows that the right-hand side of \eqref{eq:lem:wsp-help-1} exists -- provided that $\E(c^{-\frac{r}{p-r}})<\infty$.

In the second case, we assume that $1-\frac{rp}{d\left(p-r\right)}(s-s')>0$.
Here, we choose a suitable $\fq$ and apply once more Hölder's inequality to obtain that the right-hand side of (\ref{eq:lem:wsp-help-1}) is bounded by 
\[
\left[u\right]_{s,p,\eps,\Q,c}^{r}\left(\left(\sumeps{x\in\Q^{\eps}}\sumeps{y\in\Q^{\eps}}\left(c_{\frac{x}{\eps},\frac{y}{\eps}}\right)^{-\fq}\right)^{\frac{1}{\tilde{q}}}\left(\sumeps{x\in\Q^{\eps}}\sumeps{y\in\Q^{\eps}}\left|x-y\right|^{-d\left(1-\frac{rp}{d\left(p-r\right)}(s-s')\right)\frac{\tilde{q}}{\tilde{q}-1}}\right)^{\frac{\tilde{q}-1}{\tilde{q}}}\right)^{\frac{p-r}{p}}\,,
\]
where $\tilde{q}:=\fq\frac{p-r}{r}>1$. The limit $\eps\to0$
of the right-hand side exists if and only if $\E(c^{-\fq})<\infty$ and
\[
1>\left(1-\frac{rp}{d\left(p-r\right)}(s-s')\right)\frac{\tilde{q}}{\tilde{q}-1}\qquad\Leftrightarrow\qquad\fq>\frac{d}{p(s-s')}\,.
\]
Hence, we infer the following lemma.
\begin{lem}
\label{lem:weighted semi-poincare-general}Let $p\in(1,\infty)$ and
let $s\in(0,1)$. If $c$ satisfies Assumption \ref{assu:fq} for
some $\fq\in\left(\frac{d}{ps},+\infty\right]$,
then for all $r>1$ such that $\fq>\frac{r}{p-r}$ there exists $s'\in(0,s)$ such that (\ref{eq:lem:weighted semi-poincare-general})
holds uniformly in $\eps>0$.\end{lem}
\begin{proof}
Let us first assume that $\frac{r}{p-r}>\frac{d}{ps}$.
It follows that there exists $s'\in(0,s)$ such that \eqref{eq:lem:wsp-case-1} is fulfilled and therefore the right-hand side of \eqref{eq:lem:wsp-help-1} stays bounded if $\E(c^{-\frac{r}{p-r}})<\infty$.
This, however, is clearly the case due to Assumption \ref{assu:fq} and since $\fq>\frac{r}{p-r}$.

Let us now assume that $\frac{r}{p-r}\leq\frac{d}{ps}$.
In this case, there exists $s'\in(0,s)$ such that $1-\frac{rp}{d\left(p-r\right)}(s-s')>0$ and therefore the claim of the lemma follows by the second case that we have considered above. 
\end{proof}
Combined with Theorem \ref{thm:discr-Poincare-wsp-Qeps}, we obtain
the following result as a consequence of Lemma \ref{lem:weighted semi-poincare-general}. 

\begin{thm}
\label{thm:discr-compact-weights}Let $p\in(1,\infty)$, let $s\in(0,1)$
and let $c$ satisfy Assumption \ref{assu:fq} for some $\fq\in\left(\frac{d}{ps},+\infty\right]$.
If $\Q\subset\Rd$ is a uniform extension domain, then for every $r^{\star}<p_{\fq}^{\star}=\frac{dp\fq}{2d+d\fq-sp\fq}$ there exists
$C>0$, which does not depend on $\eps$, such that 
\[
\left\Vert u\right\Vert _{L^{r^{\star}}(\Q^{\eps})}\leq C\left(\sumeps{x\in\Q^{\eps}}\sumeps{y\in\Q^{\eps}}c_{\frac{x}{\eps},\frac{y}{\eps}}\frac{\left|u(x)-u(y)\right|^{p}}{\left|x-y\right|^{d+ps}}\right)^{\frac{1}{p}}\,.
\]
Moreover, for every sequence $\ue\in W^{s,p}(\Q^{\eps},c)$ such that
$\sup_{\eps>0}\left\Vert \ue\right\Vert _{s,p,\eps,\Q,c}<\infty$,
the sequence $\cR_{\eps}^{\ast}\ue$ is precompact in $L^{r^{\star}}(\Q)$.

Finally, if $\Q$ is a bounded $C^{0,1}$-domain and  $\ue\in W_0^{s,p}(\Q^{\eps},c)$ such that
$\sup_{\eps>0}\left\Vert \ue\right\Vert _{s,p,\eps,\Q,c}<\infty$,
the sequence $\cR_{\eps}^{\ast}\ue$ is precompact in $L^{r^{\star}}(\Q)$.
\end{thm}
\begin{proof}
Note that Theorem \ref{thm:discr-Poincare-wsp-Qeps} and Lemma
\ref{lem:weighted semi-poincare-general} imply that for every $r>1$ such that $\fq>\frac{r}{p-r}$ and $s'\in(0,s)$
such that $\fq>\frac{d}{p(s-s')}$ and $r^{\star}:=dr/(d-s'r)$ there
exists a constant $C>0$, which
does not depend on $\eps$, such that 
\[
\left\Vert u\right\Vert _{L^{r^{\star}}(\Q^{\eps})}
\leq \left(\sumeps{x\in\Q^{\eps}}\sumeps{y\in\Q^{\eps}}\frac{\left|u(x)-u(y)\right|^{r}}{\left|x-y\right|^{d+rs'}}\right)^{\frac{1}{r}}
\stackrel{\eqref{eq:lem:weighted semi-poincare-general}}\leq C\left(\sumeps{x\in\Q^{\eps}}\sumeps{y\in\Q^{\eps}}c_{\frac{x}{\eps},\frac{y}{\eps}}\frac{\left|u(x)-u(y)\right|^{p}}{\left|x-y\right|^{d+ps}}\right)^{\frac{1}{p}}\,,
\]
and the claimed compactness holds. It only remains to verify that
$r^{\star}$ can take any value up to $dp\fq/(2d+d\fq-sp\fq)$. Let us note the
following equivalences
\begin{align*}
\fq & >\frac{d}{p(s-s')} &  & \Leftrightarrow & s' & <s-\frac{d}{p\fq}\,,\\
\fq & >\frac{r}{p-r} &  & \Leftrightarrow & r & <\frac{p\fq}{1+\fq}\,,
\end{align*}
and that we can chose $s'$ and $r$ arbitrarily close to their upper
bounds. Hence, we obtain from $r^{\star}=dr/(d-s'r)$ that 
\[
r^{\star}< \frac{dp\fq}{1+\fq}\left(d-\frac{p\fq}{1+\fq}\left(s-\frac{d}{p\fq}\right)\right)^{-1}
\]
and $r^{\star}$ can take any value between $1$ and the right-hand
side. A short calculation shows that this is the claim.

\end{proof}

\section{\label{sec:Proof-of-Theorems}Proof of Theorems \ref{thm:Main Theorem-I}
to \ref{thm:Main Theorem-V}}

\subsection{Auxiliary Lemmas}

We recall the following useful lemma.
\begin{lem}
\label{lem:liminf-estim-convex-part}Let $\Q\subset\Rd$ be a bounded
domain and 
\MH{let $v:\R^m\to\R$ be non-negative and continuous.} 
Let $u_{\eps}:\Zde\to\R^m$
a sequence of functions having support in $\Q^{\eps}$ such that $\cR_{\eps}^{\ast}\ue\to u$
pointwise a.e. Then 
\[
\liminf_{\eps\to0}\sumeps{x\in\Q\cap\Zde}v(\ue)\geq\int_{\Q}v(u(x))\d x\,.
\]

\end{lem}
The proof is simple. However, we provide it here as a preparation
for the more involved proofs that will follow. 
\begin{proof}
\MH{For simplicity, restrict to $m=1$.} Without loss of generality, we may assume 
\[
\mathrm{E}_{\infty}:=\liminf_{\eps\to0}\sumeps{x\in\Q\cap\Zde}v(\ue)<+\infty\,.
\]
For $M\in\N$ we denote $\ue^{M}:=\max\left\{ -M,\min\left\{ \ue,M\right\} \right\} $,
i.e. the function $\ue$ is cut to values in the interval $[-M,M]$. We then
note that $\cR_{\eps}^{\ast}\ue^{M}\to u^{M}$ pointwise a.e. Using
this insight and continuity, we obtain 
\[
\mathrm{E}_{\infty}\geq\liminf_{\eps\to0}I_{M,\eps}\,,\quad\mbox{where}\quad I_{M,\eps}=\sumeps{x\in\Q\cap\Zde}v(\ue^{M})\,.
\]
From continuity of $v$ and Lebesgue's dominated convergence theorem,
we infer $I_{M,\eps}\to\int_{\Q}v(u^{M}(x))\d x$ as $\eps\to0$.
Now, we infer from Fatou's lemma that 
\[
\int_{\Q}v(u(x))\d x\leq\int_{\Q}\liminf_{M\to\infty}v(u^{M}(x))\d x\leq E_{\infty}
\]
and hence the lemma is proved.
\end{proof}
A related lemma is the following.
\begin{lem}
\label{lem:recovery-convex-part}Let $\Q\subset\Rd$ be a bounded
domain and 
\MH{let $v:\R^m\to\R$ be non-negative and continuous} 
such that for
some $\alpha>0$ and $r'>1$ we have $v(\xi)\leq\alpha|\xi|^{r'}$.
Let $u_{\eps}:\Zde\to\R^m$ a sequence of functions having support in
$\Q^{\eps}$ such that for some $r\geq r'$ it holds $\sup_{\eps>0}\left\Vert \ue\right\Vert _{L^{r}(\Q^{\eps})}<\infty$
and $\cR_{\eps}^{\ast}\ue\to u$ pointwise a.e. Then 
\[
\lim_{\eps\to0}\sumeps{x\in\Q\cap\Zde}v(\ue)=\int_{\Q}v(u(x))\d x\,.
\]
\end{lem}
\begin{proof}
We have for some positive constant $C$ that 
\[
\limsup_{\eps\to0}\sumeps{x\in\Q\cap\Zde}\left|v(\ue)\right|<C\sup_{\eps>0}\left\Vert \ue\right\Vert _{L^{r}(\Q^{\eps})}^{r'}<+\infty\,.
\]
Let $\delta>0$. By Egorov's theorem there exists $\Q_{\delta}\subset\Q$
with $\Q_{\delta}^{\complement}=\Q\backslash\Q_{\delta}$, $|\Q_{\delta}|<\delta$
and such that $\cR_{\eps}^{\ast}\ue\to u$ uniformly on $\Q_{\delta}^{\complement}$.
Hence we have 
\begin{align*}
\lim_{\eps\to0}\sumeps{x\in\Q\cap\Zde}v(\ue)-\int_{\Q}v(u(x))\d x & =\lim_{\eps\to 0} \int_{\Q_{\delta}}\left(v(u(x))-v(\cR_{\eps}^{\ast}\ue(x))\right)\d x\\
 & \leq2\sup_{\eps>0}\int_{\Q_{\delta}}\alpha\left|\cR_{\eps}^{\ast}\ue(x)\right|^{r'}\d x\\
 & \leq2|\Q_{\delta}|^{\frac{r-r'}{r}}\alpha\sup_{\eps>0}\left\Vert \ue\right\Vert _{L^{r}(\Q^{\eps})}
\end{align*}
and hence the lemma is proved as $\delta$ becomes arbitrary small.
\end{proof}
Another important result connected with convex functions is the following.
\begin{lem}
\label{lem:convergence-convex-convol}Let $G:\R\to\R^m$ be non-negative
and convex, let $u:\Rd\to\R^m$ be measurable and such that $\int_{\Rd}G\left(u(z)\right)\d z<\infty$
and let $\left(\eta_{k}\right)_{k\in\N}$ be as in Lemma \ref{lem:Dirac-Approx}.
Then 
\[
\lim_{k\to\infty}\int_{\Rd}G(\eta_{k}\ast u)=\int_{\Rd}G\left(u\right)\,.
\]
\end{lem}
\begin{proof}
We note that $\eta_{k}(x-z)\d z$ induces a probability measure on
$\Rd$ for every $k\in\N$ and every $x\in\Rd$. Hence we infer in
a first step by Jensen's inequality
\begin{align*}
\int_{\Rd}G(\eta_{k}\ast u) & =\int_{\Rd}G\left(\int_{\Rd}\eta_{k}(x-z)u(z)\d z\right)\d x\\
 & \leq\int_{\Rd}\int_{\Rd}\eta_{k}(x-z)G\left(u(z)\right)\d z\d x\\
 & \leq\int_{\Rd}G\left(u(z)\right)\d z\,.
\end{align*}
On the other hand, Fatou's Lemma yields 
\[
\int_{\Rd}G\left(u(z)\right)\d z\leq\liminf_{k\to\infty}\int_{\Rd}G\left(\left(u\ast\eta_{k}\right)(z)\right)\d z\,.
\]

\end{proof}
The following lemma is new to our knowledge. It is the basis for the
proofs of our main results.
\begin{lem}
\label{lem:liminf-estim-Wsp}Let $\Q\subset\Rd$ be a bounded domain
and let $c$, $s$, $p$, $\fq$ and $V:\R^m\to\R$ satisfy Assumptions \ref{assu:fq}
and \ref{assu:V}. Furthermore, let $u_{\eps}:\Zde\to\R^m$ a sequence
of functions having support in $\Q^{\eps}$ such that $\cR_{\eps}^{\ast}\ue\to u$
pointwise a.e. and $\sup_{\eps}\left\Vert \ue\right\Vert _{\infty}<\infty$.
Then 
\begin{align}
\liminf_{\eps\to0}\sumsumeps{(x,y)\in\Ztde}c_{\frac{x}{\eps},\frac{y}{\eps}}\frac{V\left(\ue(x)-\ue(y)\right)}{\left|x-y\right|^{d+ps}} & \geq\E(c)\underset{{\scriptstyle \Rd\times\Rd}}{\iint}\frac{V\left(u(x)-u(y)\right)}{\left|x-y\right|^{d+ps}}\d x\d y\,.\label{eq:conv-lemma-1}\\
\liminf_{\eps\to0}\sumsumeps{(x,y)\in\Q^{\eps}\times\Q^{\eps}}c_{\frac{x}{\eps},\frac{y}{\eps}}\frac{V\left(\ue(x)-\ue(y)\right)}{\left|x-y\right|^{d+ps}} & \geq\E(c)\underset{{\scriptstyle \Q\times\Q}}{\iint}\frac{V\left(u(x)-u(y)\right)}{\left|x-y\right|^{d+ps}}\d x\d y\,.\label{eq:conv-lemma-2}
\end{align}
Furthermore, if 
\begin{align}
\sup_{\eps>0}\sup_{x,y\in\Zde}\frac{\left|\ue(x)-\ue(y)\right|}{\left|x-y\right|} & =:C_{L}<\infty\label{eq:conv-lemma-2-a}\\
\mbox{resp. }\quad\sup_{\eps>0}\sup_{x,y\in\Q^{\eps}}\frac{\left|\ue(x)-\ue(y)\right|}{\left|x-y\right|} & =:C_{L}<\infty\,,\label{eq:conv-lemma-2-b}
\end{align}
and $u\in C_{c}^{1}(\Rd)$, resp. $u\in W^{1,\infty}(\Q)$ then we
have 
\begin{align}
\lim_{\eps\to0}\sumsumeps{(x,y)\in\Ztde}c_{\frac{x}{\eps},\frac{y}{\eps}}\frac{V\left(\ue(x)-\ue(y)\right)}{\left|x-y\right|^{d+ps}} & =\E(c)\underset{{\scriptstyle \Rd\times\Rd}}{\iint}\frac{V\left(u(x)-u(y)\right)}{\left|x-y\right|^{d+ps}}\d x\d y\,.\label{eq:conv-lemma-3}\\
\lim_{\eps\to0}\sumsumeps{(x,y)\in\Q^{\eps}\times\Q^{\eps}}c_{\frac{x}{\eps},\frac{y}{\eps}}\frac{V\left(\ue(x)-\ue(y)\right)}{\left|x-y\right|^{d+ps}} & =\E(c)\underset{{\scriptstyle \Q\times\Q}}{\iint}\frac{V\left(u(x)-u(y)\right)}{\left|x-y\right|^{d+ps}}\d x\d y\,.\label{eq:conv-lemma-4}
\end{align}
\end{lem}
\begin{proof}
We only prove (\ref{eq:conv-lemma-1}) and (\ref{eq:conv-lemma-3})
and shortly discuss how to generalize the calculations to (\ref{eq:conv-lemma-2}) and
(\ref{eq:conv-lemma-4}). Without loss of generality, we assume that
the $\liminf$ of the left hand side of (\ref{eq:conv-lemma-1}) is
bounded. For each $0<\xi<R<\infty$ the sum 
\[
\sumsumeps{(x,y)\in\Ztde}c_{\frac{x}{\eps},\frac{y}{\eps}}\frac{V\left(\ue(x)-\ue(y)\right)}{\left|x-y\right|^{d+ps}}
\]
can be split into the three sums over $(x,y)\in\Zde\times\Zde$ such
that either $\left\{ |x-y|<\xi\right\} $, $\left\{ \xi\leq|x-y|<R\right\} $
or $\left\{ |x-y|\geq R\right\} $. We denote the corresponding sums
by $I_{\xi}^{\eps}$, $I_{\xi,R}^{\eps}$ and $I_{R}^{\eps}$. In
what follows, we prove in three steps that 
\begin{align}
I_{\xi,R}^{\eps} & =\sumsumeps{\substack{(x,y)\in\Ztde \\
\xi\leq\left|x-y\right|<R}
}c_{\frac{x}{\eps},\frac{y}{\eps}}\frac{V\left(\ue(x)-\ue(y)\right)}{\left|x-y\right|^{d+ps}}\to\E(c)\underset{{\scriptstyle \xi\leq\left|x-y\right|<R}}{\iint}\frac{V\left(u(x)-u(y)\right)}{\left|x-y\right|^{d+ps}}\,,\label{eq:main-thm-lim-xiR}\\
I_{R}^{\eps} & =\sumsumeps{\substack{ (x,y)\in\Ztde\\
\left|x-y\right|\geq R
}}c_{\frac{x}{\eps},\frac{y}{\eps}}\frac{V\left(\ue(x)-\ue(y)\right)}{\left|x-y\right|^{d+ps}}\to O(R^{-ps})\,,\label{eq:main-thm-lim-R}\\
I_{\xi}^{\eps} & =\sumsumeps{\substack{(x,y)\in\Ztde\\
\left|x-y\right|<\xi}
}c_{\frac{x}{\eps},\frac{y}{\eps}}\frac{V\left(\ue(x)-\ue(y)\right)}{\left|x-y\right|^{d+ps}}\to O(\xi^{p-ps})\,,\label{eq:main-thm-lim-xi}
\end{align}
where we show (\ref{eq:main-thm-lim-xi}) only in case (\ref{eq:conv-lemma-2-a})
holds. Without loss of generality, we will thereby assume that $R>2\mathrm{diam}(\Q)$.
In what follows, we prove (\ref{eq:main-thm-lim-xiR})-(\ref{eq:main-thm-lim-xi})
in 3 steps. This provides (\ref{eq:conv-lemma-3}) on observing that
\[
\underset{{\scriptstyle \left|x-y\right|<\xi}}{\iint}\frac{V\left(u(x)-u(y)\right)}{\left|x-y\right|^{d+ps}}\leq\left\Vert \nabla u\right\Vert _{\infty}\underset{\begin{gathered}{\scriptstyle (x,y)\in2\Q\times2\Q}\\
{\scriptstyle \left|x-y\right|<\xi}
\end{gathered}
}{\iint}\frac{\left|x-y\right|^{p}}{\left|x-y\right|^{d+ps}}=O(\xi^{p-ps})
\]
and 
\[
\limsup_{\eps\to0}\sumsumeps{(x,y)\in\Ztde}c_{\frac{x}{\eps},\frac{y}{\eps}}\frac{V\left(\ue(x)-\ue(y)\right)}{\left|x-y\right|^{d+ps}}\leq\limsup_{\eps\to0}\left(I_{\xi,R}^{\eps}+I_{\xi}^{\eps}+I_{R}^{\eps}\right)\,.
\]
Inequality (\ref{eq:conv-lemma-1}) can be proved on noting that 
\begin{align*}
V_{\infty}: & =\liminf_{\eps\to0}\sumsumeps{(x,y)\in\Ztde}c_{\frac{x}{\eps},\frac{y}{\eps}}\frac{V\left(\ue(x)-\ue(y)\right)}{\left|x-y\right|^{d+ps}}\\
 & \geq\sup_{\xi,R}\,\,\E(c)\underset{{\scriptstyle \xi\leq\left|x-y\right|<R}}{\iint}\frac{V\left(u(x)-u(y)\right)}{\left|x-y\right|^{d+ps}}
\end{align*}
and applying the Beppo Levi monotone convergence theorem as $\xi\to0$
and $R\to\infty$. 

We note that (\ref{eq:conv-lemma-2}) and (\ref{eq:conv-lemma-4})
can be proved in the same way using some slight modification. In particular,
we replace $V\left(\ue(x)-\ue(y)\right)$ by $$\tilde{V}\left(x,y,\ue(x)-\ue(y)\right):=\chi_{\Q}(x)\chi_{\Q}(y)V\left(\ue(x)-\ue(y)\right)$$
and study 
\begin{align*}
I_{\xi,R}^{\eps} & =\sumsumeps{\substack{(x,y)\in B^{\eps}\times B^{\eps}\\
\xi\leq\left|x-y\right|}
}c_{\frac{x}{\eps},\frac{y}{\eps}}\frac{V\left(x,y,\ue(x)-\ue(y)\right)}{\left|x-y\right|^{d+ps}}\to\E(c)\underset{\begin{gathered}{\scriptstyle (x,y)\in B\times B}\\
{\scriptstyle \xi\leq\left|x-y\right|}
\end{gathered}
}{\iint}\frac{V\left(x,y,u(x)-u(y)\right)}{\left|x-y\right|^{d+ps}}\,,\\
I_{\xi}^{\eps} & =\sumsumeps{\substack{(x,y)\in B^{\eps}\times B^{\eps}\\
\left|x-y\right|<\xi}
}c_{\frac{x}{\eps},\frac{y}{\eps}}\frac{V\left(x,y,\ue(x)-\ue(y)\right)}{\left|x-y\right|^{d+ps}}\to O(\xi^{p-ps})\,,
\end{align*}
where $B=B(0)$ is an open ball around $0$ that contains $\Q$ and
$B^{\eps}:=B\cap\Zde$.

\textbf{Step 1:} We consider the lower semi-continuous extension $g_{\eps}(x,y):=\frac{V\left(\ue(x)-\ue(y)\right)}{\max\left\{ \left|x-y\right|^{d+2s},\,\xi^{d+2s}\right\} }$
to the set $\left|x-y\right|<R$. Moreover, since $R>2\mathrm{diam}(\Q)$,
it holds $V\left(\ue(x)-\ue(y)\right)=0$ if $|x|\geq2R$ or $|y|\geq2R$.
Hence, the support of $g_{\eps}$ is a compact convex subset of $\left|x-y\right|<R$
and we infer from Theorem \ref{thm:FHS-Ergodic-Thm} that 
\[
\sumsumeps{\substack{(x,y)\in\Ztde\\
\left|x-y\right|<R}
}c_{\frac{x}{\eps},\frac{y}{\eps}}g_{\eps}(x,y)=\E(c)\underset{{\scriptstyle \left|x-y\right|<R}}{\iint}g(x,y)\,,
\]
where $g(x,y):=\frac{V\left(u(x)-u(y)\right)}{\max\left\{ \left|x-y\right|^{d+2s},\,\xi^{d+2s}\right\} }$.
Since the same arguments hold on the set $\left|x-y\right|<\xi$,
the limit (\ref{eq:main-thm-lim-xiR}) holds.

\textbf{Step 2:} Due to our assumption on $R$, we find $\left|x-y\right|>R$
implies that at most one of the points $x,y$ lies in $\Q$. Since
$u$ vanishes outside $\Q$, we obtain by a symmetrization 
\[
I_{R}^{\eps}\leq\left\Vert V(\ue(\cdot))\right\Vert _{\infty}\sumsumeps{\substack{
R<\left|x-y\right|\\
x\in\Q\cap\Zde}
}\left(c_{\frac{x}{\eps},\frac{y}{\eps}}+c_{\frac{y}{\eps},\frac{x}{\eps}}\right)\frac{1}{\left|x-y\right|^{d+ps}}\,.
\]
For simplicity of notation, we write $\tilde{c}_{\frac{x}{\eps},\frac{y}{\eps}}:=c_{\frac{x}{\eps},\frac{y}{\eps}}+c_{\frac{y}{\eps},\frac{x}{\eps}}$
which is an ergodic variable with $\E(\tilde{c})=2\E(c)$. We denote
$\calR_{k,R}^{\eps}:=\left\{ z\in\Zde\,|\;2^{k}R<\left|z\right|\leq2^{k+1}R\right\} $
and reformulate $I_{R}^{\eps}$ as 
\[
I_{R}^{\eps}=\sumsumeps{\substack{R<\left|x-y\right|\\
x\in\Q\cap\Zde}
}\tilde{c}_{\frac{x}{\eps},\frac{y}{\eps}}\frac{1}{\left|x-y\right|^{d+ps}}=\sumeps{x\in\Q\cap\Zde}\sum_{k=0}^{\infty}\sumeps{y\in\calR_{k,R}^{\eps}+x}\tilde{c}_{\frac{x}{\eps},\frac{y}{\eps}}\frac{1}{\left|x-y\right|^{d+ps}}\,,
\]
which we estimate as 
\begin{align*}
I_{R}^{\eps} & \leq\sum_{k=0}^{\infty}\left(2^{k}R\right)^{-ps}\sumeps{x\in(1+k)\Q\cap\Zde}\left(2^{k}R\right)^{-d}\sumeps{\left|x-y\right|<2^{k+1}R}\tilde{c}_{\frac{x}{\eps},\frac{y}{\eps}}\\
 & =\left|\Q\right|\sum_{k=0}^{\infty}\left(2^{k}R\right)^{-ps}\left(1+k\right)^{d}\left|\mathbb{B}_{\Q,k,R}\right|^{-1}\sumsumeps{(x,y)\in\mathbb{B}_{\Q,k,R}^{\eps}}\tilde{c}_{\frac{x}{\eps},\frac{y}{\eps}}\,,
\end{align*}
where $\mathbb{B}_{\Q,k,R}=\left\{ (x,y)\in\Rd\,|\;x\in(1+k)\Q\,,\,\left|x-y\right|\leq2^{k+1}R\right\} $
and $\mathbb{B}_{\Q,k,R}^{\eps}:=\mathbb{B}_{\Q,k,R}\cap\Zde$. Defining
$\fc_{R,k,\eps}:=\eps^{2d}\left|\mathbb{B}_{\Q,k,R}\right|^{-1}\sum_{(x,y)\in\mathbb{B}_{\Q,k,R}^{\eps}}\tilde{c}_{\frac{x}{\eps},\frac{y}{\eps}}$
we infer from Lemma \ref{lem:ergodic-bound} an estimate $\fc_{R}:=\sup_{\eps,k}\fc_{R,k,\eps}<\infty$
and boundedness of 
\[
I_{R}^{\eps}\leq\fc_{R}R^{-ps/2}\frac{\left|\Q\right|\E(\tilde{c})}{1-2^{-ps/2}}\sup_{k}\left(2^{-kps/2}(1+k)^{d}\right)\,.
\]

\textbf{Step 3:} Now let $\sup_{\eps>0}\sup_{x,y\in\Zde}\frac{\left|\ue(x)-\ue(y)\right|}{\left|x-y\right|}=:C_{L}<\infty$.
In order to treat the remaining term $I_{\xi}^{\eps}$, note that
Assumption \ref{assu:V} implies uniform boundedness and lower semi-continuity
of the function $\widetilde{V}_{\eps}(x,y):=\frac{V\left(\ue(x)-\ue(y)\right)}{\left|\ue(x)-\ue(y)\right|^{p}}$.
Furthermore, if either $\mathrm{dist}(x,\Q)>\xi$ or $\mathrm{dist}(y,\Q)>\xi$
then $\left|x-y\right|<\xi$ implies $\widetilde{V}_{\eps}(x,y)=0$. Hence
we obtain 
\begin{align*}
I_{\xi}^{\eps} & \leq\left\Vert \widetilde{V}\right\Vert _{\infty}\sumsumeps{\substack{x\in(2\Q)\cap\Zde\\
\left|x-y\right|<\xi}
}\tilde{c}_{\frac{x}{\eps},\frac{y}{\eps}}\frac{\left|\ue(x)-\ue(y)\right|^{p}}{\left|x-y\right|^{d+ps}}\\
 & \leq C_{L}\left\Vert \widetilde{V}\right\Vert _{\infty}\sumsumeps{\substack{x\in(2\Q)\cap\Zde\\
\left|x-y\right|<\xi}
}\tilde{c}_{\frac{x}{\eps},\frac{y}{\eps}}\left|x-y\right|^{-d+p(1-s)}\,,
\end{align*}
and (\ref{eq:main-thm-lim-xi}) follows from Lemma \ref{lem:ergodic-bound-polynomial}. \end{proof}

\subsection{\label{sub:Proof-of-Theorem I and II}Proof of Theorems \ref{thm:Main Theorem-I}
and \ref{thm:Main Theorem-III}--\ref{thm:Main Theorem-V}}

\MH{For simplicity of notation, we restrict to $m=1$.} We will only prove Theorem \ref{thm:Main Theorem-I}. Theorems \ref{thm:Main Theorem-III}--\ref{thm:Main Theorem-V}
can be proved in the same way replacing Theorem \ref{thm:Compactness-discrete-eps-to-0} by the Embedding Theorems \ref{thm:embedding-dirichlet}
and \ref{thm:embedding-mean}.

\paragraph{Proof of Part 1.}

Since $\cR_{\eps}^{\ast}f_{\eps}\weakto f$ weakly in $L^{r^{\ast}}(\Q)$,
we find $\sup_{\eps}\left\Vert f_{\eps}\right\Vert _{L^{r^{\ast}}(\Q_{\eps})}<\infty$.
Thus, we find from the scaled young inequality for every $\delta>0$
some $C_{\delta}$ such that 
\[
\left|\sumeps{x\in\Q^{\eps}}f_{\eps}(x)\ue(x)\right|\leq\left\Vert u\right\Vert _{L^{r}(\Q^{\eps})}\left\Vert f_{\eps}\right\Vert _{L^{r^{\ast}}(\Q_{\eps})}\leq\delta\left\Vert u\right\Vert _{L^{r}(\Q^{\eps})}^{p}+C_{\delta}\left\Vert f_{\eps}\right\Vert _{L^{r^{\ast}}(\Q_{\eps})}^{p^{\ast}}\,,
\]
where $\frac{1}{p}+\frac{1}{p^{\ast}}=1$ and $r=\frac{r^\ast}{r^\ast-1}$. Since $r<p_\fq^\star$, we find from Theorem
\ref{thm:discr-compact-weights} that 
\begin{align*}
\left\Vert u\right\Vert _{L^{r}(\Q^{\eps})}^{p} & \leq\sup_{\eps}\left(\sumsumeps{(x,y)\in\Ztde}c_{\frac{x}{\eps},\frac{y}{\eps}}\frac{V\left(\ue(x)-\ue(y)\right)}{\left|x-y\right|^{d+ps}}\right)+\sumeps{x\in\Zde}G(u(x))\\
 & \leq\sup_{\eps}\sE_{p,s,\eps}(\ue)+\delta\left\Vert u\right\Vert _{L^{r}(\Q^{\eps})}^{p}+C_{\delta}\left\Vert f_{\eps}\right\Vert _{L^{r^{\ast}}(\Q_{\eps})}^{p^{\ast}}
\end{align*}
implying (for suitable choice of $\delta$) boundedness of 
\[
\sumsumeps{(x,y)\in\Ztde}c_{\frac{x}{\eps},\frac{y}{\eps}}\frac{V\left(\ue(x)-\ue(y)\right)}{\left|x-y\right|^{d+ps}}+\sumeps{x\in\Zde}G(u(x))\,.
\]
In particular, we obtain that 
\[
\mathrm{E}_{\infty}:=\liminf_{\eps\to0}\sumsumeps{(x,y)\in\Ztde}c_{\frac{x}{\eps},\frac{y}{\eps}}\frac{V\left(\ue(x)-\ue(y)\right)}{\left|x-y\right|^{d+ps}}<+\infty
\]
is bounded.

\MH{In case of Theorem \ref{thm:Main Theorem-I}, since $\Q$ is bounded and $\ue$ is $0$ outside $\Q$, we can assume w.l.o.g. that $\Q$ is cubic and hence a uniform extension domain.}
From Assumption \ref{assu:fq} and Theorems \ref{thm:Compactness-discrete-eps-to-0} and \ref{thm:discr-compact-weights}
it follows that $\sup_{\eps>0}\left\Vert \cR_{\eps}^{\ast}\ue\right\Vert _{L^{r}}<\infty$
and the existence of $u\in L^{r}(\Q)$ such that $\cR_{\eps}^{\ast}\ue\to u$
strongly in $L^{r}(\Q)$ and pointwise a.e. along a subsequence $\eps'\to0$.
Furthermore, for $M\in\N$ we denote $\ue^{M}:=\max\left\{ -M,\min\left\{ \ue,M\right\} \right\} $,
the function $\ue$ cut to values in the interval $[-M,M]$. We then
note that $\cR_{\eps}^{\ast}\ue^{M}\to u^{M}$ strongly in $L^{r}(\Q)$
and pointwise a.e. Using this insight, we obtain using Lemma \ref{lem:liminf-estim-Wsp}
that 
\begin{align*}
\mathrm{E}_{\infty} & \geq\liminf_{\eps\to0}\sumsumeps{(x,y)\in\Ztde}
c_{\frac{x}{\eps},\frac{y}{\eps}}\frac{V\left(\ue^{M}(x)-\ue^{M}(y)\right)}{\left|x-y\right|^{d+ps}}\\
 & \geq\E(c)\iint_{\Rd\times\Rd}\frac{V\left(u^{M}(x)-u^{M}(y)\right)}{\left|x-y\right|^{d+ps}}\d x\,\d y
\end{align*}
Since the above considerations hold for every $M$, we apply Fatou's
Lemma (resp. the monotone convergence theorem by Beppo-Levi) and find
\[
\E(c)\underset{{\scriptstyle \R^{2d}}}{\iint}\frac{V\left(u(x)-u(y)\right)}{\left|x-y\right|^{d+ps}}\d x\,\d y\leq\liminf_{M\to\infty}\E(c)\iint_{\Rd\times\Rd}\frac{V\left(u^{M}(x)-u^{M}(y)\right)}{\left|x-y\right|^{d+ps}}\d x\,\d y\leq\mathrm{E}_{\infty}\,.
\]
Moreover, we have from Lemma \ref{lem:liminf-estim-convex-part} that
\[
\liminf_{\eps\to0}\sumeps{x\in\Zde}G(\ue(x))-\sumeps{x\in\Zde}\ue(x)f_{\eps}(x)\geq\int_{\Rd}G(u(x))\d x-\int_{\Rd}u(x)f(x)\,.
\]

\subsubsection*{Proof of Part 2}

We first consider $u\in C_{c}^1(\Q)$. In this case, we set $\ue(x)=u(x)$
for $x\in\Zde$. From Lemma \ref{lem:liminf-estim-Wsp} we infer 
\[
\lim_{\eps\to0}\sumsumeps{(x,y)\in\Ztde}c_{\frac{x}{\eps},\frac{y}{\eps}}\frac{V\left(\ue(x)-\ue(y)\right)}{\left|x-y\right|^{d+ps}}=\E(c)\iint_{\Rd\times\Rd}\frac{V\left(u(x)-u(y)\right)}{\left|x-y\right|^{d+ps}}\d x\,\d y\,.
\]
Now, let $\sE_{p,s}(u)<\infty$ with $u(x)=0$ outside of $\Q$, set
$\eps_{0}=1$. By Assumption \ref{assu:V}, we find $u\in W^{s,p}(\Rd)$
and in particular, there exists a sequence $u_{k}\in C_{c}^{1}(\Q)$
such that $u_{k}\to u$ in $W^{s,p}(\Rd)$. Moreover, since $m<p_{\fq}^{\star}\leq p^{\star}$, 
Lemma \ref{lem:recovery-convex-part} yields $\int_{\Rd}G(u_{k})\to\int_{\Rd}G(u)$
and hence 
\[
\lim_{k\to\infty}\sE_{p,s}(u_{k})=\sE_{p,s}(u)\,.
\]
From the above calculation, there exists $\eps_{k}>0$ such that for
all $\eps<\eps_{k}$, $\left|\sE_{p,s,\eps}(u_{k})-\sE_{p,s}(u_{k})\right|<\left|\sE_{p,s}(u)-\sE(\eta_{k}\ast u)\right|$
and in total 
\[
\left|\sE_{p,s,\eps}(u_{k})-\sE_{p,s}(u)\right|<2\left|\sE_{p,s}(u)-\sE(\eta_{k}\ast u)\right|\,.
\]
Setting $u^{\eps}:=u_{k}$ for all $\eps\in[\eps_{k+1},\eps_{k})$,
(\ref{eq:Thm-Gamma-limsup}) holds.

\subsection{Proof of Theorem \ref{thm:Main Theorem-II}}

{For simplicity of notation, we restrict to $m=1$.} The proof mostly follows the lines of Section \ref{sub:Proof-of-Theorem I and II}.
However, as there are a few modifications due to the non-boundedness
of the domain, we provide the full proof for completeness.

\paragraph{Proof of Part 1.}

Since $f\in C_{c}(\Rd)$, we chose some bounded domain $\Q$ such
that $f$ has its support in $\Q$. From here, we may follow the lines
of Section \ref{sub:Proof-of-Theorem I and II} to obtain boundedness
of 
\[
\sumsumeps{(x,y)\in\Ztde}c_{\frac{x}{\eps},\frac{y}{\eps}}\frac{V\left(\ue(x)-\ue(y)\right)}{\left|x-y\right|^{d+ps}}+\sumeps{x\in\Zde}G(u(x))\,.
\]
In particular, we obtain that 
\[
\mathrm{E}_{\infty}:=\liminf_{\eps\to0}\sumsumeps{(x,y)\in\Ztde}c_{\frac{x}{\eps},\frac{y}{\eps}}\frac{V\left(\ue(x)-\ue(y)\right)}{\left|x-y\right|^{d+ps}}<+\infty
\]
is bounded.

Now, let $m\in\N$ and consider $B_{m}:=\left\{ x\in\Rd\,:\;|x|<m\right\} $.
From Assumption \ref{assu:fq} and Theorem \ref{thm:discr-compact-weights}
it follows that $\sup_{\eps>0}\left\Vert \cR_{\eps}^{\ast}\ue\right\Vert _{L^{r}(B_{m})}<\infty$
and the existence of $u_{m}\in L^{r}(B_{m})$ and a subsequence $\eps_{m}$
such that $\cR_{\eps_{m}}^{\ast}u_{\eps_{m}}\to u_{m}$ as $\eps_{m}\to0$
strongly in $L^{r}(B_{m})$ and pointwise a.e. in $B_{m}$. Furthermore,
for $M\in\N$ we denote $\ue^{M}:=\max\left\{ -M,\min\left\{ \ue,M\right\} \right\} $
and obtain using Lemma \ref{lem:liminf-estim-Wsp} that 
\begin{align*}
\mathrm{E}_{\infty} & \geq\liminf_{\eps\to0}\sumsumeps{(x,y)\in B_{m}^{\eps}\times B_{m}^{\eps}}
c_{\frac{x}{\eps},\frac{y}{\eps}}\frac{V\left(u_{\eps_{m}}^{M}(x)-u_{\eps_{m}}^{M}(y)\right)}{\left|x-y\right|^{d+ps}}\\
 & \geq\E(c)\iint_{B_{m}\times B_{m}}\frac{V\left(u_{m}^{M}(x)-u_{m}^{M}(y)\right)}{\left|x-y\right|^{d+ps}}\d x\,\d y
\end{align*}
Since the above considerations hold for every $M$, we apply Fatous
Lemma (resp. the monotone convergence theorem by Beppo-Levi) and find
\[
\E(c)\iint_{B_{m}\times B_{m}}\frac{V\left(u_{m}(x)-u_{m}(y)\right)}{\left|x-y\right|^{d+ps}}\d x\,\d y\leq\liminf_{M\to\infty}\E(c)\iint_{B_{m}\times B_{m}}\frac{V\left(u_{m}^{M}(x)-u_{m}^{M}(y)\right)}{\left|x-y\right|^{d+ps}}\d x\,\d y\leq\mathrm{E}_{\infty}\,.
\]
Using a Cantor argument, we infer the existence of a measurable $u:\,\Rd\to\R$
such that $\cR_{\eps'}^{\ast}u_{\eps'}\to u$ pointwise a.e. along
a subsequence $\eps'\to0$ and the Fatou Lemma yields 
\[
\E(c)\iint_{\Rd\times\Rd}\frac{V\left(u(x)-u(y)\right)}{\left|x-y\right|^{d+ps}}\d x\,\d y\leq\liminf_{m\to\infty}\E(c)\iint_{B_{m}\times B_{m}}\frac{V\left(u_{m}(x)-u_{m}(y)\right)}{\left|x-y\right|^{d+ps}}\d x\,\d y\leq\mathrm{E}_{\infty}\,.
\]
Moreover, we have from Lemma \ref{lem:liminf-estim-convex-part} that
\[
\liminf_{\eps\to0}\sumeps{x\in\Zde}G(\ue(x))-\sumeps{x\in\Zde}\ue(x)f_{\eps}(x)\geq\int_{\Rd}G(u(x))\d x-\int_{\Rd}u(x)f(x)\,.
\]

\subsubsection*{Proof of Part 2}

We first consider $u\in C_{c}^1(\Q)$. In this case, we set $\ue(x)=u(x)$
for $x\in\Zde$. From Lemma \ref{lem:liminf-estim-Wsp} we infer 
\[
\lim_{\eps\to0}\sumsumeps{(x,y)\in\Ztde}c_{\frac{x}{\eps},\frac{y}{\eps}}\frac{V\left(\ue(x)-\ue(y)\right)}{\left|x-y\right|^{d+ps}}=\E(c)\iint_{\Rd\times\Rd}\frac{V\left(u(x)-u(y)\right)}{\left|x-y\right|^{d+ps}}\d x\,\d y\,.
\]
Now, let $\sE_{p,s}(u)<\infty$ with $u(x)=0$ outside of $\Q$, set
$\eps_{0}=1$. By Assumption \ref{assu:V}, we find $u\in W^{s,p}(\Rd)$
and in particular, by Lemma \ref{lem:Dirac-Approx} we infer $u_{k}:=\eta_{k}\ast u\to u$
in $W^{s,p}(\Rd)$. Moreover, Lemma \ref{lem:convergence-convex-convol}
yields $\int_{\Rd}G(\eta_{k}\ast u)\to\int_{\Rd}G(u)$ and hence 
\[
\lim_{k\to\infty}\sE_{p,s}(u_{k})=\sE_{p,s}(u)\,.
\]
From the above calculation, there exists $\eps_{k}>0$ such that for
all $\eps<\eps_{k}$, $\left|\sE_{p,s,\eps}(u_{k})-\sE_{p,s}(u_{k})\right|<\left|\sE_{p,s}(u)-\sE(\eta_{k}\ast u)\right|$
and in total 
\[
\left|\sE_{p,s,\eps}(u_{k})-\sE_{p,s}(u)\right|<2\left|\sE_{p,s}(u)-\sE_{p,s}(\eta_{k}\ast u)\right|\,.
\]
Setting $u^{\eps}:=u_{k}$ for all $\eps\in[\eps_{k+1},\eps_{k})$,
(\ref{eq:Thm-Gamma-limsup}) holds.

\appendix

\section{Proofs of Auxiliary results}

\subsection{Proof of Lemma \ref{lem:Dirac-Approx}}
\begin{lem}
\label{lem:Wsp-shifts}Let $u\in W^{s,p}(\Rd)$. Then 
\begin{equation}
\lim_{h\to0}\left\Vert u(\cdot)-u(\cdot-h)\right\Vert _{s,p}\to0\,.\label{eq:lem:Wsp-shifts}
\end{equation}
\end{lem}
\begin{proof}
It is well known that  
\[
\lim_{h\to0}\left\Vert u(\cdot)-u(\cdot-h)\right\Vert _{L^{p}(\Rd)}\to0
\]
and it only remains to show 
\[
\lim_{h\to0}\left[u(\cdot)-u(\cdot-h)\right]_{s,p}\to0\,.
\]
Suppose $u\in C_{c}^{\infty}(\Rd)$ and let $B$ be a ball that contains
the support of $u$. We write $u_{h}(x):=u(x-h)$ as well as $f(x,y)=u(x)-u(y)$
and similarly $f_{h}(x,y)$. Since, for small $h$, $f(x,y)=f_{h}(x,y)=0$ if both $x,y\not\in2B$,
we observe that
\begin{align*}
 & \int_{\Rd}\int_{\Rd}\frac{\left|f(x,y)-f_{h}(x,y)\right|^{p}}{\left|x-y\right|^{d+sp}}\,\d x\,\d y\\
 & \qquad\qquad=\int_{2B}\int_{2B}\frac{\left|f(x,y)-f_{h}(x,y)\right|^{p}}{\left|x-y\right|^{d+sp}}\,\d x\,\d y+2\int_{2B}\int_{\Rd\backslash2B}\frac{\left|f(x,y)-f_{h}(x,y)\right|^{p}}{\left|x-y\right|^{d+sp}}\,\d x\,\d y\\
 & \qquad\qquad\leq2\int_{2B}\int_{\Rd}\frac{\left|u(x)-u_{h}(x)-u(y)+u_{h}(y)\right|^{p}}{\left|x-y\right|^{d+sp}}\,\d x\,\d y\,.
\end{align*}
For every $\delta>0$ the right-hand side can be split into an integral
over $$A_{\delta}:=\left\{ (x,y)\,:\;x\in2B,\,\left|x-y\right|<\xi\right\} $$
and the complement. We find
\begin{align*}
\left[u(\cdot)-u(\cdot-h)\right]_{s,p} & \leq2^{p+1}\int_{A_{\delta}}\frac{\left|u_{h}(x)-u_{h}(y)\right|^{p}+\left|u(x)-u(y)\right|^{p}}{\left|x-y\right|^{d+sp}}\\
 & \quad+2\int_{\R^{2d}\backslash A_{\delta}}\frac{\left|u(x)-u_{h}(x)-u(y)+u_{h}(y)\right|^{p}}{\left|x-y\right|^{d+sp}}\,\d x\,\d y\,.
\end{align*}
The first integral can be estimated by 
\[
2^{p+2}\left\Vert \nabla u\right\Vert _{\infty}^{p}\int_{A_{\delta}}\frac{1}{\left|x-y\right|^{d+sp-p}}=2^{p+2}\left\Vert \nabla u\right\Vert _{\infty}^{p}|2B|\,|S^{d-1}|\,\delta^{p-sp}\,.
\]
The second integral converges to $0$ as $h\to0$ as it is bounded
by 
\[
\delta^{-d-sp}4\left\Vert u-u_{h}\right\Vert \to0\,.
\]
Hence, we have shown that $\lim_{h\to0}\left[u(\cdot)-u(\cdot-h)\right]_{s,p}\leq C\delta^{p-sp}$
for every $\delta>0$, implying (\ref{eq:lem:Wsp-shifts}). For arbitrary
$u\in W^{s,p}(\Rd)$ the lemma follows from a standard approximation
argument.\end{proof}
\begin{rem}
\label{rem:lem:Wsp-shift}Via the triangle inequality, the last lemma
implies that $h\mapsto\left\Vert u(\cdot)-u(\cdot-h)\right\Vert _{s,p}$
is continuous:
\[
\left|\left\Vert u(\cdot)-u(\cdot-h_{1})\right\Vert _{s,p}-\left\Vert u(\cdot)-u(\cdot-h_{2})\right\Vert _{s,p}\right|\leq\left\Vert u(\cdot-h_{1})-u(\cdot-h_{2})\right\Vert _{s,p}\,.
\]
\end{rem}
\begin{proof}[Proof of Lemma \ref{lem:Dirac-Approx}]
First note that it is well known that 
\[
\left\Vert u\ast\eta_{k}\right\Vert _{L^{p}(\Rd)}\leq\left\Vert u\right\Vert _{L^{p}(\Rd)}\qquad\mbox{and}\qquad\lim_{k\to\infty}\left\Vert u\ast\eta_{k}-u\right\Vert _{L^{p}(\Rd)}=0
\]
and it only remains to show 
\[
\left[u\ast\eta_{k}\right]_{s,p}\leq\left[u\right]_{s,p}\qquad\mbox{and}\qquad\lim_{k\to\infty}\left[u\ast\eta_{k}-u\right]_{s,p}=0\,.
\]
The inequality can be easily verified from the fact that 
\begin{align*}
 & \int_{\Rd}\int_{\Rd}\frac{\left|\int_{\Rd}\left(\eta_{k}(z)u(x-z)-\eta_{k}(z)u(y-z)\right)\d z\right|^{p}}{\left|x-y\right|^{d+sp}}\,\d x\,\d y\\
 & \qquad\leq\left\Vert \eta_{k}\right\Vert _{L^{1}(\Rd)}^{p/p^{\ast}}\int_{\Rd}\eta_{k}(z)\int_{\Rd}\int_{\Rd}\frac{\left|\left(u(x-z)-u(y-z)\right)\right|^{p}}{\left|x-y\right|^{d+sp}}\,\d x\,\d y\,\d z\\
 & \qquad=\int_{\Rd}\int_{\Rd}\frac{\left|\left(u(x)-u(y)\right)\right|^{p}}{\left|x-y\right|^{d+sp}}\,\d x\,\d y\,.
\end{align*}
The limit behavior follows from Lemma \ref{lem:Wsp-shifts}, Remark
\ref{rem:lem:Wsp-shift} and the following calculation: 
\begin{align*}
 & \int_{\Rd}\int_{\Rd}\frac{\left|\int_{\Rd}\left(\eta_{k}(z)\left(u(x-z)-u(x)\right)-\eta_{k}(z)\left(u(y-z)-u(y)\right)\right)\d z\right|^{p}}{\left|x-y\right|^{d+sp}}\,\d x\,\d y\\
 & \qquad\leq\left\Vert \eta_{k}\right\Vert _{L^{1}(\Rd)}^{p/p^{\ast}}\int_{\Rd}\eta_{k}(z)\int_{\Rd}\int_{\Rd}\frac{\left|\left(u(x-z)-u(x)-u(y-z)+u(y)\right)\right|^{p}}{\left|x-y\right|^{d+sp}}\,\d x\,\d y\,\d z\\
 & \qquad\leq\left\Vert \eta_{k}\right\Vert _{L^{1}(\Rd)}^{p/p^{\ast}}\int_{\Rd}\eta_{k}(z)\left[u(\cdot)-u(\cdot-z)\right]_{s,p}\d z\\
 & \qquad\to0\quad\mbox{as }k\to\infty\,.
\end{align*}

\end{proof}

\subsection{Proof of Remark \ref{rem:Extention-bounded-support}}

The remark is a consequence of the following lemma.
\begin{lem}
\label{lem:prod-wsp-c1}Let $\varphi\in C_{c}^{1}(\Rd)$. Then for
every $\eps>0$, $p\in(1,\infty)$, $s\in(0,1)$ and $u\in W^{s,p}(\Zde)$
it holds $\varphi u\in W^{s,p}(\Zde)$ and there exists some $C>0$
which does not depend on $\eps$ such that 
\begin{equation}
\left\Vert \varphi u\right\Vert _{s,p,\eps}\leq C\left\Vert u\right\Vert _{s,p,\eps}\left\Vert \varphi\right\Vert _{C_{0}^{1}(\Rd)}\,.\label{eq:lem:prod-wsp-c1}
\end{equation}
\end{lem}
\begin{proof}
Writing $\delta_{f}\left(x,y\right):=\left|f(x)-f(y)\right|$, we
first observe that 
\[
\sumeps{x\in\Zde}\sumeps{y\in\Zde}\frac{\delta_{u\varphi}\left(x,y\right)^{p}}{\left|x-y\right|^{d+ps}}\leq\sumeps{x\in\Zde}\sumeps{y\in\Zde}\frac{\left|u(x)\right|\delta_{\varphi}\left(x,y\right)+\delta_{u}\left(x,y\right)\left|\varphi(y)\right|}{\left|x-y\right|^{d+ps}}\delta_{u\varphi}\left(x,y\right)^{p-1}\,.
\]
Let $B(x):=\left\{ y\in\Rd\,:\;|x-y|<1\right\} $ with complement
$B^{\complement}(x)$. Then, for every $x\in\Zde$ we find 
\begin{align*}
\sumeps{y\in\Zde}\frac{\delta_{\varphi}\left(x,y\right)^{p}}{\left|x-y\right|^{d+ps}} & \leq\sumeps{y\in B(x)\cap\Zde}\frac{\left\Vert \nabla\varphi\right\Vert _{\infty}^{p}}{\left|x-y\right|^{d+ps-p}}+\sumeps{y\in B^{\complement}(x)\cap\Zde}\frac{\left\Vert \varphi\right\Vert _{\infty}^{p}}{\left|x-y\right|^{d+ps}}\\
 & \leq C\left(\left\Vert \nabla\varphi\right\Vert _{\infty}^{p}+\left\Vert \varphi\right\Vert _{\infty}^{p}\right)\,.
\end{align*}
Furthermore, note that
\begin{align*}
 & \sumeps{x\in\Zde}\sumeps{y\in\Zde}\frac{\left|u(x)\right|\delta_{\varphi}\left(x,y\right)}{\left|x-y\right|^{d+ps}}\left|\delta_{u\varphi}\left(x,y\right)\right|^{p-1}\\
 & \qquad\qquad\leq\sumeps{x\in\Zde}\left|u(x)\right|\left(\sumeps{y\in\Zde}\frac{\delta_{\varphi}\left(x,y\right)^{p}}{\left|x-y\right|^{d+ps}}\right)^{\frac{1}{p}}\left(\sumeps{y\in\Zde}\frac{\left|\delta_{u\varphi}\left(x,y\right)\right|^{p}}{\left|x-y\right|^{d+ps}}\right)^{\frac{p-1}{p}}\\
 & \qquad\qquad\leq C\left(\left\Vert \nabla\varphi\right\Vert _{\infty}^{p}+\left\Vert \varphi\right\Vert _{\infty}^{p}\right)\left\Vert u\right\Vert _{L^{p}(\Zde)}\left(\sumeps{x\in\Zde}\sumeps{y\in\Zde}\frac{\left|\delta_{u\varphi}\left(x,y\right)\right|^{p}}{\left|x-y\right|^{d+ps}}\right)^{\frac{p-1}{p}}
\end{align*}
as well as 
\begin{align*}
\sumeps{x\in\Zde}\sumeps{y\in\Zde}\frac{\delta_{u}\left(x,y\right)\left|\varphi(y)\right|}{\left|x-y\right|^{d+ps}}\left|\delta_{u\varphi}\left(x,y\right)\right|^{p-1} & \leq\left\Vert \varphi\right\Vert _{\infty}^{p}\left[u\right]_{s,p,\eps}\left(\sumeps{x\in\Zde}\sumeps{y\in\Zde}\frac{\left|\delta_{u\varphi}\left(x,y\right)\right|^{p}}{\left|x-y\right|^{d+ps}}\right)^{\frac{p-1}{p}}\,.
\end{align*}
Hence we obtain (\ref{eq:lem:prod-wsp-c1}).
\end{proof}

\subsection{Proof of Theorem \ref{thm:discr-Poincare-wsp}}
We prove Theorem \ref{thm:discr-Poincare-wsp} after three auxiliary lemmas. The first lemma is an equivalent to Lemma 6.1 in \cite{DiNezza2012}.
\begin{lem}
\label{lem:discr-isoperimetric-ineq}Let $1\leq p<\infty$, $s\in(0,1)$.
There exists $C$ depending only on $s$, $p$ and $d$ such that
\[
\sumeps{y\in E^{\complement}\cap\Zde}\frac{1}{\left|x-y\right|^{d+sp}}\geq C\left|E\right|^{-sp/d}
\]
for every $\eps>0$, every $x\in\Zde$ and every measurable set $E\subset\Rd$
with finite measure.
\end{lem}
\begin{proof}
Let $\rho:=2d^{\frac{1}{d}}\left(\left|E\right|_{\eps}\right)^{\frac{1}{d}}$,  where $\left|E\right|_{\eps}=\eps^{d}\sharp\left\{ E\cap\Zde\right\}$, see the beginning of Section \ref{sec:Preliminaries}.
Then, for every $\tilde{\rho}\geq\rho$ we find $\left|B_{\tilde{\rho}}(x)\right|_{\eps}\geq\left|E\right|_{\eps}$
and hence 
\begin{align*}
\left|E^{\complement}\cap B_{\tilde{\rho}}(x)\right|_{\eps} & =\left|B_{\tilde{\rho}}(x)\right|_{\eps}-\left|E\cap B_{\tilde{\rho}}(x)\right|_{\eps}\geq\left|E\right|_{\eps}-\left|E\cap B_{\tilde{\rho}}(x)\right|_{\eps}\\
 & \geq\left|E\cap B_{\tilde{\rho}}^{\complement}(x)\right|_{\eps}\,.
\end{align*}
Hence, we infer that
\begin{align*}
\sumeps{y\in E^{\complement}\cap\Zde}\frac{1}{\left|x-y\right|^{d+sp}} & =\sumeps{y\in E^{\complement}\cap B_{\tilde{\rho}}(x)}\frac{1}{\left|x-y\right|^{d+sp}}+\sumeps{y\in E^{\complement}\cap B_{\tilde{\rho}}^{\complement}(x)}\frac{1}{\left|x-y\right|^{d+sp}}\\
 & \geq\frac{\left|E^{\complement}\cap B_{\tilde{\rho}}(x)\right|_{\eps}}{\left|\tilde{\rho}\right|^{d+sp}}+\sumeps{y\in E^{\complement}\cap B_{\tilde{\rho}}^{\complement}(x)}\frac{1}{\left|x-y\right|^{d+sp}}\\
 & \geq\frac{\left|E\cap B_{\tilde{\rho}}^{\complement}(x)\right|_{\eps}}{\left|\tilde{\rho}\right|^{d+sp}}+\sumeps{y\in E^{\complement}\cap B_{\tilde{\rho}}^{\complement}(x)}\frac{1}{\left|x-y\right|^{d+sp}}\\
 & \geq\sumeps{y\in E\cap B_{\tilde{\rho}}^{\complement}(x)}\frac{1}{\left|x-y\right|^{d+sp}}+\sumeps{y\in E^{\complement}\cap B_{\tilde{\rho}}^{\complement}(x)}\frac{1}{\left|x-y\right|^{d+sp}}\\
 & =\sumeps{y\in B_{\tilde{\rho}}^{\complement}(x)}\frac{1}{\left|x-y\right|^{d+sp}}\, .
\end{align*}
 Next, we consider cells $C_{\eps}(z):=z+\eps(-\frac{1}{2},\frac{1}{2})$,
$z\in\Zde\backslash\{0\}$. On each of these cells, we want to estimate the ratio between the maximal and the minimal value of the function $f(y)=|y|^{-d-ps}$. Due to the polynomial decay of this function, the closer one of the cells $C_\eps(z)$ lies next to $0$, the higher will be the ratio in $f$. The biggest value that $f$ can attain on $\Rd\backslash C_\eps(0)$ is $\eps^{-d-ps}$. Furthermore, all neighboring cells to $C_\eps(0)$ lie within the cube $\left(-\frac32\eps,\frac32\eps\right)$ and the minimal value of $f$ is on this domain is  the value of $f$ is $\left(\frac{3}{2}d^{\frac{1}{d}}\eps\right)^{-d-sp}$. Hence we obtain 

\[
\inf_{z\in\Zde\backslash\{0\}}\inf_{y\in C_{\eps}(z)}\left|y\right|^{-d-sp}\left(\sup_{y\in C_{\eps}(z)}\left|y\right|^{-d-sp}\right)^{-1}\geq\left(\frac{\eps}{2}\right)^{d+sp}\left(\frac{3}{2}d^{\frac{1}{d}}\eps\right)^{-d-sp}=\left(3d^{\frac{1}{d}}\right)^{-d-sp}\,,
\]

and we conclude that
\[
\sumeps{y\in E^{\complement}\cap\Zde}\frac{1}{\left|x-y\right|^{d+sp}}\geq\sumeps{y\in B_{\tilde{\rho}}^{\complement}(x)}\frac{1}{\left|x-y\right|^{d+sp}}\geq\left(3d^{\frac{1}{d}}\right)^{-d-sp}\int_{y\in B_{\tilde{\rho}}^{\complement}(x)}\frac{1}{\left|x-y\right|^{d+sp}}\d y\,.
\]
Now the theorem follows from integration using polar coordinates.\end{proof}

\begin{lem}[{\cite[Lemma 6.2]{DiNezza2012}}]
\label{lem:resort-sum} Let $s\in(0,1)$ and $p\in[1,\infty)$ be
such that $sp<d$. Fix $T>1$ and let $N\in\Z$,
and 
\[
a_{k}\mbox{ be a non-increasing sequence such that }a_{k}=0\mbox{ for every }k\geq N\,.
\]
Then, 
\[
\sum_{k\in\Z}a_{k}^{(d-sp)/d}T^{k}\leq C\sum_{k\in\Z,\,a_{k}\not=0}a_{k+1}a_{k}^{-sp/d}T^{k}\,,
\]
for a suitable constant $C=C(d,s,p,T)$, independent of $N$.
\end{lem}
We are now in the position to prove the following variant
of \cite{DiNezza2012}, Lemma 6.3.
\begin{lem}
\label{lem:poincare-help-3}Let $s\in(0,1)$ and $p\in[1,\infty)$
be such that $sp<d$. Let $f\in L^{\infty}(\Zde)$ be compactly supported.
For any $k\in\Z$ let 
\[
a_{k}:=\left|\left\{ \left|f\right|>2^{k}\right\} \right|\,.
\]
Then, 
\[
\sumeps{x\in\Zde}\sumeps{y\in\Zde}\frac{\left|f(x)-f(y)\right|^{p}}{\left|x-y\right|^{d+ps}}\geq C\sum_{a_k \neq 0} 2^{pk}a_{k+1}a_{k}^{-sp/d}
\]
for some suitable constant $C=C(d,s,p)>0$, which depends not on $\eps$.\end{lem}
\begin{proof}
We first emphasize that $\left|\left|f(x)\right|-\left|f(y)\right|\right|\leq\left|f(x)-f(y)\right|$
and hence we only consider $f\geq0$, possibly replacing $f$ by $\left|f\right|$.

We define
\begin{align*}
A_{k} & :=\left\{ f>2^{k}\right\} \quad\mbox{with}\quad A_{k+1}\subset A_{k}\\
a_{k} & :=\left|\left\{ f>2^{k}\right\} \right|\quad\mbox{with}\quad a_{k+1}\leq a_{k}\, .
\end{align*}
We define
\begin{align*}
 & D_{k}:=A_{k}\backslash A_{k+1}=\left\{ 2^{k+1}\geq f>2^{k}\right\} \quad\mbox{and}\quad d_{k}:=\left|D_{k}\right|\quad\mbox{with}\\
 & d_{k}\mbox{ and }a_{k}\mbox{ are bounded and they become zero when }k\mbox{ is large enough,}
\end{align*}
since $f$ is bounded. We define $D_{-\infty}=\left\{ f=0\right\} $
and further observe that the sets $D_{k}$ are mutually disjoint and
\begin{equation}
D_{-\infty}\cup\bigcup_{l\in\Z,l\leq k}D_{l}=A_{k+1}^{\complement}\,,\qquad\bigcup_{l\in\Z,l\geq k}D_{l}=A_{k}\,.\label{eq:lemma-pI-help-1}
\end{equation}
As a consequence, we have
\begin{equation}
a_{k}=\sum_{l=k}^{\infty}d_{l}\,,\qquad d_{k}=a_{k}-\sum_{l=k+1}^{\infty}d_{l}\,.\label{eq:lemma-pI-help-2}
\end{equation}
The first equality implies that the series $\sum_{l\geq k}d_{l}$
are convergent. For convenience of notation, in the following we write for arbitrary expressions $g(y)$
\[
\sum_{j=-\infty}^{i-2}\sumeps{y\in D_{j}}g(y):=\sumeps{y\in D_{-\infty}}g(y)+\sum_{\substack{l\in\Z\\
l\leq i-2
}
}\sumeps{y\in D_{j}}g(y)
\]
Now, we fix $i\in\Z$ and $x\in D_{i}$. For every $j\in\Z$, $j\leq i-2$
and every $y\in D_{j}$ we have
\[
\left|f(x)-f(y)\right|\geq2^{i}-2^{j+1}\geq2^{i}-2^{i-1}=2^{i-1}
\]
and hence by the first equality in (\ref{eq:lemma-pI-help-1}) it
holds
\begin{align*}
\sum_{j=-\infty}^{i-2}\sumeps{y\in D_{j}}\frac{\left|f(x)-f(y)\right|^{p}}{\left|x-y\right|^{d+sp}} & \geq2^{p(i-1)}\sum_{j=-\infty}^{i-2}\sumeps{y\in D_{j}}\frac{1}{\left|x-y\right|^{d+sp}}\\
 & \geq2^{p(i-1)}\sumeps{y\in A_{i-1}^{\complement}}\frac{1}{\left|x-y\right|^{d+sp}}\,.
\end{align*}
Therefore, by Lemma \ref{lem:discr-isoperimetric-ineq}, there exists a
constant $c_{0}$ such that for every $i\in\Z$ and every $x\in D_{i}$
it holds
\begin{align*}
\sum_{j=-\infty}^{i-2}\sumeps{y\in D_{j}}\frac{\left|f(x)-f(y)\right|^{p}}{\left|x-y\right|^{d+sp}} & \geq c_{0}2^{pi}a_{i-1}^{-sp/d}\,,\\
\mbox{and}\qquad\sum_{j=-\infty}^{i-2}\sumeps{y\in D_{j}}\sumeps{x\in D_{i}}\frac{\left|f(x)-f(y)\right|^{p}}{\left|x-y\right|^{d+sp}} & \geq c_{0}2^{pi}a_{i-1}^{-sp/d}d_{i}\,.
\end{align*}
Summing up the last inequality over $i\in\Z$ we have on one side
\begin{equation}
\sum_{\substack{i\in\Z\\
a_{i-1}\not=0
}
}\sum_{j=-\infty}^{i-2}\sumeps{y\in D_{j}}\sumeps{x\in D_{i}}\frac{\left|f(x)-f(y)\right|^{p}}{\left|x-y\right|^{d+sp}}\geq c_{0}\sum_{l\in\Z,\,a_{l-1}\not=0}2^{pl}a_{l-1}^{-sp/d}d_{l}=:c_{0}S\,,\label{eq:lemma-pI-help-3}
\end{equation}
implying $S$ to be bounded, and on the other hand, using (\ref{eq:lemma-pI-help-2}),
we have 
\begin{align}
 & \sum_{\substack{i\in\Z\\
a_{i-1}\not=0
}
}\sum_{j=-\infty}^{i-2}\sumeps{y\in D_{j}}\sumeps{x\in D_{i}}\frac{\left|f(x)-f(y)\right|^{p}}{\left|x-y\right|^{d+sp}}\nonumber \\
 & \qquad\geq c_{0}\sum_{\substack{i\in\Z\\
a_{i-1}\not=0
}
}\left(2^{pi}a_{i-1}^{-sp/d}a_{i}-\sum_{l=i+1}^{\infty}2^{pi}a_{i-1}^{-sp/d}d_{l}\right)\,,\label{eq:lemma-pI-help-4}
\end{align}
where we estimate the second sum by $S$ through 
\begin{align*}
\sum_{\substack{i\in\Z\\
a_{i-1}\not=0
}
}\sum_{\substack{l\in\Z\\
l\geq i+1
}
}2^{pi}a_{i-1}^{-sp/d}d_{l} & =\sum_{\substack{i\in\Z\\
a_{i-1}\not=0
}
}\sum_{\substack{l\in\Z\\
l\geq i+1\\
a_{i-1}d_{l}\not=0
}
}2^{pi}a_{i-1}^{-sp/d}d_{l}\\
 & \leq\sum_{\substack{i\in\Z}
}\sum_{\substack{l\in\Z\\
l\geq i+1\\
a_{l-1}\not=0
}
}2^{pi}a_{i-1}^{-sp/d}d_{l}\\
 & =\sum_{\substack{l\in\Z\\
a_{l-1}\not=0
}
}\sum_{\substack{i\in\Z\\
i\leq l-1
}
}2^{pi}a_{i-1}^{-sp/d}d_{l}\\
 & \leq\sum_{\substack{l\in\Z\\
a_{l-1}\not=0
}
}\sum_{\substack{i\in\Z\\
i\leq l-1
}
}2^{pi}a_{l-1}^{-sp/d}d_{l}\leq S\,.
\end{align*}
Using the last estimate in (\ref{eq:lemma-pI-help-4}), we obtain
\[
\sum_{\substack{i\in\Z\\
a_{i-1}\not=0
}
}\sum_{j=-\infty}^{i-2}\sumeps{y\in D_{j}}\sumeps{x\in D_{i}}\frac{\left|f(x)-f(y)\right|^{p}}{\left|x-y\right|^{d+sp}}\geq c_{0}\sum_{\substack{i\in\Z\\
a_{i-1}\not=0
}
}2^{pi}a_{i-1}^{-sp/d}a_{i}-c_{0}S
\]
and using estimate (\ref{eq:lemma-pI-help-3}) we find upon relabeling
$c_{0}$ that 
\[
\sum_{\substack{i\in\Z\\
a_{i-1}\not=0
}
}\sum_{j=-\infty}^{i-2}\sumeps{y\in D_{j}}\sumeps{x\in D_{i}}\frac{\left|f(x)-f(y)\right|^{p}}{\left|x-y\right|^{d+sp}}\geq c_{0}\sum_{\substack{i\in\Z\\
a_{i-1}\not=0
}
}2^{pi}a_{i-1}^{-sp/d}a_{i}\,.
\]
On the other hand, it clearly holds that 
\[
\sumeps{y\in\Zde}\sumeps{x\in\Zde}\frac{\left|f(x)-f(y)\right|^{p}}{\left|x-y\right|^{d+sp}}\geq\sum_{\substack{i\in\Z\\
a_{i-1}\not=0
}
}\sum_{j=-\infty}^{i-2}\sumeps{y\in D_{j}}\sumeps{x\in D_{i}}\frac{\left|f(x)-f(y)\right|^{p}}{\left|x-y\right|^{d+sp}}
\]
and hence the lemma follows.
\end{proof}
We are now in the position to prove the first Sobolev theorem.
\begin{proof}[Proof of Theorem \ref{thm:discr-Poincare-wsp}]
 It suffices to prove the claim for 
\[
\left[f\right]_{s,p,\eps}^{p}=\sumeps{x\in\Zde}\sumeps{y\in\Zde}\frac{\left|f(x)-f(y)\right|^{p}}{\left|x-y\right|^{d+ps}}<\infty
\]
and for $f\in L^{\infty}(\Zde)$. Indeed, for arbitrary $f\in W^{s,p}(\Zde)$,
with $f_{N}:=\max\left\{ -N,\min\left\{ N,f\right\} \right\} $ we obtain
that 
\[
\lim_{N\to\infty}\sumeps{x\in\Zde}\sumeps{y\in\Zde}\frac{\left|f_{N}(x)-f_{N}(y)\right|^{p}}{\left|x-y\right|^{d+ps}}=\sumeps{x\in\Zde}\sumeps{y\in\Zde}\frac{\left|f(x)-f(y)\right|^{p}}{\left|x-y\right|^{d+ps}}
\]
due to the dominated convergence theorem and pointwise convergence
$f_{N}\to f$.

We recall the definitions 
\begin{align*}
A_{k} & :=\left\{ |f|>2^{k}\right\} \quad\mbox{with}\quad A_{k+1}\subset A_{k}\\
a_{k} & :=\left|\left\{ |f|>2^{k}\right\} \right|\quad\mbox{with}\quad a_{k+1}\leq a_{k}
\end{align*}
from the proof of Lemma \ref{lem:poincare-help-3} and obtain
\[
\left\Vert f\right\Vert _{L^{p^{\ast}}(\Zde)}^{p^{\star}}=\sum_{k\in\Z}\sumeps{x\in A_{k}\backslash A_{k+1}}\left|f(x)\right|^{p^{\star}}\leq\sum_{k\in\Z}\sumeps{x\in A_{k}\backslash A_{k+1}}\left|2^{k+1}\right|^{p^{\star}}\leq\sum_{k\in\Z}2^{(k+1)p^{\star}}a_{k}\,.
\]
Using $p/p^{\star}=(d-sp)/d=1-sp/d<1$ we can conclude with Lemma \ref{lem:resort-sum} that
\begin{align*}
\left\Vert f\right\Vert _{L^{p^{\ast}}(\Zde)}^{p} & \leq2^{p}\left(\sum_{k\in\Z}2^{kp^{\star}}a_{k}\right)^{\frac{p}{p^{\star}}}\leq2^{p}\,\sum_{k\in\Z}2^{kp}a_{k}^{(d-sp)/d}\\
 & \leq C\sum_{\substack{k\in\Z\\
a_{k}\not=0
}
}2^{kp}a_{k+1}a_{k}^{\frac{-sp}{d}}\,.
\end{align*}
It only remains to apply Lemma \ref{lem:poincare-help-3} and relabeling
the constant $C$ to find (\ref{eq:discr-poinc-zde}) in case $q=p^{\star}$.
In case $q=\theta p+(1-\theta)p^{\star}$, $\theta\in(0,1)$, we obtain
from Hölder's inequality and the case $q=p^{\star}$ that
\begin{align*}
\sumeps{x\in\Zde}\left|f(x)\right|^{q} & =\sumeps{x\in\Zde}\left|f(x)\right|^{\theta p}\left|f(x)\right|^{(1-\theta)p^{\star}}\leq\left(\sumeps{x\in\Zde}\left|f(x)\right|^{p}\right)^{\theta}\left(\sumeps{x\in\Zde}\left|f(x)\right|^{p^{\star}}\right)^{1-\theta}\\
 &=\left\Vert f\right\Vert _{L^{p}(\Zde)}^{p\theta}\left\Vert f\right\Vert _{L^{p^{\star}}(\Zde)}^{\left(1-\theta\right)p^{\star}}\leq\left\Vert f\right\Vert _{L^{p}(\Zde)}^{p\theta}\left[f\right]_{s,p,\eps}^{\left(1-\theta\right)p^{\star}}\leq\left\Vert f\right\Vert _{s,p,\eps}^{p\theta}\left\Vert f\right\Vert _{s,p,\eps}^{\left(1-\theta\right)p^{\star}}=\left\Vert f\right\Vert _{s,p,\eps}^{q}\,.
\end{align*}

\end{proof}

\subsection{Proof of Theorem \ref{thm:Compactness-discrete-eps-to-0}}
\begin{proof}
Since $\Q$ is a uniform extension domain, the family $\cR_{\eps}^{\ast}u^{\eps}$
is precompact if and only if $\cR_{\eps}^{\ast}\cE_{\eps}u^{\eps}$ is compact, where we recall the operator $\cE_{\eps}$ from Definition \ref{def:uniform-extension-domain}.
We will apply the Frechet-Kolmogorov(-Riesz) theorem to prove compactness
of $\cR_{\eps}^{\ast}\cE_{\eps}u^{\eps}$. More precisely, it suffices
to verify the following three properties: 
\begin{align}
 & \sup_{\eps>0}\left\Vert \cR_{\eps}^{\ast}\cE_{\eps}u^{\eps}\right\Vert _{L^{q}(\Rd)}<\infty\,,\qquad\lim_{R\to\infty}\sup_{\eps>0}\left\Vert \cR_{\eps}^{\ast}\cE_{\eps}u^{\eps}\right\Vert _{L^{q}(\Rd\backslash B_{R}(0))}=0\,,\label{eq:FKR-conditions-1-2}\\
 & \lim_{|h|\to0}\sup_{\eps>0}\left\Vert \cR_{\eps}^{\ast}\cE_{\eps}u^{\eps}(\cdot)-\cR_{\eps}^{\ast}\cE_{\eps}u^{\eps}(\cdot+h)\right\Vert _{L^{q}(\Rd)}\to0\,.\label{eq:FKR-conditions-3}
\end{align}
Note that the conditions in (\ref{eq:FKR-conditions-1-2}) are satisfied
due to Theorem \ref{thm:discr-Poincare-wsp} and Remark \ref{rem:Extention-bounded-support}.
Thus, it only remains to show (\ref{eq:FKR-conditions-3}).

For $h\in\Rd$ we write $\tau_{h}u(x):=u(x+h)$, whenever this is
well defined. Moreover, for every $\eps>0$ we define 
\[
\left\Vert u\right\Vert _{p,\eps}:= \left( \sumeps{x\in\Zde}\left|u(x)\right|^{p} \right)^{1/p}\, .
\]
We first prove the Theorem in case $q=p$. Let $h\in\Zde$ and $\eta:=10h$.
We define $B_{\eta,\eps}:=\left\{ y\in\Zde\,:\;|y|\leq|\eta|\right\} $
and $B_{\eta}:=\left\{ y\in\Rd\,:\;|y|<\left|\eta\right|\right\} $.
Since $h\in\Zde$ we always have $\eta\geq10\eps$ and hence we have
\[
C_{B}:=\sup_{\eps,\eta}\left(\frac{\left|B_{\eta,\eps}\right|_{\eps}}{\left|B_{\eta}\right|}+\frac{\left|B_{\eta}\right|}{\left|B_{\eta,\eps}\right|_{\eps}}\right)<+\infty
\]
and
\[
\tilde{C}_{B}:=\sup_{\eps,\eta}\left(\sumeps{y\in B_{\eta,\eps}}\left|y\right|^{\left(d+ps\right)/\left(p-1\right)}\right)/\left(\int_{B_{\eta}}\left|y\right|^{\left(d+ps\right)/\left(p-1\right)}\d y\right)<+\infty
\]
We find 
\[
\left\Vert u-\tau_{h}u\right\Vert _{p,\eps}\leq\left|B_{\eta,\eps}\right|_{\eps}^{-1}\sumeps{y\in B_{\eta,\eps}}\left(\left\Vert u-\tau_{y}u\right\Vert _{p,\eps}+\left\Vert \tau_{h}u-\tau_{y}u\right\Vert _{p,\eps}\right)\,.
\]
In order to estimate the right-hand side, we apply Hölder's inequality and obtain
\begin{align*}
\sumeps{y\in B_{\eta,\eps}}\left\Vert u-\tau_{y}u\right\Vert _{p,\eps} & \leq\left(\sumeps{y\in B_{\eta,\eps}}\frac{\left\Vert u-\tau_{y}u\right\Vert _{p,\eps}^{p}}{\left|y\right|^{d+ps}}\right)^{\frac{1}{p}}\left(\sumeps{y\in B_{\eta,\eps}}\left|y\right|^{\left(d+ps\right)/\left(p-1\right)}\right)^{\frac{p-1}{p}}\\
 & \leq\left(\sumeps{y\in B_{\eta,\eps}}\,\sumeps{x\in\cap\Zde}\frac{\left(u(x)-u(x+y)\right)^{p}}{\left|y\right|^{d+ps}}\right)^{\frac{1}{p}}\tilde{C}_{B}^{\frac{p-1}{p}}\left(\int_{B_{\eta}}\left|y\right|^{\left(d+ps\right)/\left(p-1\right)}\d y\right)^{\frac{p-1}{p}}\\
 & =C\left\Vert u\right\Vert _{W^{s,p}(\Zde)}\left|B_{\eta}\right|\left|\eta\right|^{s}\,.
\end{align*}
Also with $B_{2\eta,\eps}(h):=\left\{ y\in\Zde\,:\;|y-h|\leq2|\eta|\right\} $
we get from Hölder's inequality
\begin{align*}
\sumeps{y\in B_{\eta,\eps}}\left\Vert \tau_{h}u-\tau_{y}u\right\Vert _{p,\eps} & \leq\left(\sumeps{y\in B_{\eta,\eps}}\frac{\left\Vert \tau_{h}u-\tau_{y}u\right\Vert _{p,\eps}^{p}}{\left|y-h\right|^{d+ps}}\right)^{\frac{1}{p}}\left(\sumeps{y\in B_{\eta,\eps}}\left|y-h\right|^{\left(d+ps\right)/\left(p-1\right)}\right)^{\frac{p-1}{p}}\\
 & \leq C\left\Vert u\right\Vert _{W^{s,p}(\Zde)}\left(\sumeps{y\in B_{2\eta,\eps(h)}}\left|y-h\right|^{\left(d+ps\right)/\left(p-1\right)}\right)^{\frac{p-1}{p}}\\
 & =2^{d+ps}C\left\Vert u\right\Vert _{W^{s,p}(\Zde)}\left|B_{\eta}\right|\left|\eta\right|^{s}\,.
\end{align*}
This implies 
\begin{align}
\left\Vert u-\tau_{h}u\right\Vert _{p,\eps}\leq C\left\Vert u\right\Vert _{W^{s,p}(\Zde)}\left|h\right|^{s}\,.
\label{eq:FKR-help-1}
\end{align}
Now, let $C_{\eps}:=[-\eps,\eps]^{d}$ be the cube of size $\eps$
and let $h\in\Rd\backslash C_{\eps}$. Further, let $\Z_{\eps,h}^{d}:=\left\{ z\in\Zde\,:\;\left(z+C_{\eps}\right)\cap\left(h+C_{\eps}\right)\not=\emptyset\right\} $
and for every $z\in\Z_{\eps,h}^{d}$ let $V(z,h)=\left|\left(z+C_{\eps}\right)\cap\left(h+C_{\eps}\right)\right|$.
Then we find 
\begin{align*}
\left\Vert \cR_{\eps}^{\ast}u-\tau_{h}\cR_{\eps}^{\ast}u\right\Vert _{L^{p}(\Rd)} & \leq\sum_{z\in\Z_{\eps,h}^{d}}V(z,h)\left\Vert \cR_{\eps}^{\ast}u-\tau_{z}\cR_{\eps}^{\ast}u\right\Vert _{L^{p}(\Rd)}\\
 & =\sum_{z\in\Z_{\eps,h}^{d}}V(z,h)\left\Vert u-\tau_{z}u\right\Vert _{L^{p}(\Zde)}\\
 & \stackrel{\eqref{eq:FKR-help-1}}\leq C\sum_{z\in\Z_{\eps,h}^{d}}V(z,h)\left\Vert u\right\Vert _{W^{s,p}(\Zde)}\left|z\right|^{s}\\
 & \leq C\sum_{z\in\Z_{\eps,h}^{d}}V(z,h)\left\Vert u\right\Vert _{W^{s,p}(\Zde)}\left|2h\right|^{s}\\
 & \leq C\left\Vert u\right\Vert _{W^{s,p}(\Zde)}\left|h\right|^{s}\,.
\end{align*}
Now, let $h\in C_{\eps}$. Like above, we obtain 
\[
\left\Vert \cR_{\eps}^{\ast}u-\tau_{h}\cR_{\eps}^{\ast}u\right\Vert _{L^{p}(\Rd)}\leq C\sum_{z\in\Z_{\eps,h}^{d}}V(z,h)\left\Vert u\right\Vert _{W^{s,p}(\Zde)}\left|z\right|^{s}\,.
\]
However, this time we find $V(z,h)\to0$ uniformly and linearly in $\left|h\right|\to0$.
Hence, we have 
\[
\left\Vert \cR_{\eps}^{\ast}u-\tau_{h}\cR_{\eps}^{\ast}u\right\Vert _{L^{p}(\Rd)}\leq C\begin{cases}
\left|h\right|^{s} & \mbox{if }h\in\Rd\backslash C_{\eps}\\
|h| & \mbox{if }h\in C_{\eps}
\end{cases}\,.
\]
Since $C$ does not depend on $\eps$, we infer 
\begin{equation}
\left\Vert \cR_{\eps}^{\ast}u-\tau_{h}\cR_{\eps}^{\ast}u\right\Vert _{L^{p}(\Rd)}\leq C\begin{cases}
\left|h\right|^{s} & \mbox{if }\left|h\right|>1\\
|h| & \mbox{if }\left|h\right|\leq1
\end{cases}\,.\label{eq:FKR-conditions-3-help}
\end{equation}
This implies (\ref{eq:FKR-conditions-3}) in case $p=q$. 

In case $q<p$, we use Remark \ref{rem:Extention-bounded-support}
and let $\tilde{\Q}$ denote the common support of $\cE_{\eps}u^{\eps}$.
We then obtain by  Hölder's inequality 
\[
\left\Vert \cR_{\eps}^{\ast}\cE_{\eps}u^{\eps}-\tau_{h}\cR_{\eps}^{\ast}\cE_{\eps}u^{\eps}\right\Vert _{L^{q}(\Rd)}\leq\left|\tilde{\Q}\right|^{\frac{p-q}{p}}\left\Vert \cR_{\eps}^{\ast}\cE_{\eps}u^{\eps}-\tau_{h}\cR_{\eps}^{\ast}\cE_{\eps}u^{\eps}\right\Vert _{L^{p}(\Rd)}^{\frac{q}{p}}\,,
\]
and hence compactness by (\ref{eq:FKR-conditions-3-help}).

In case $q\in(p,p^{\star})$ we use the same trick as in the proof
of Theorem \ref{thm:discr-Poincare-wsp}: we have for $f=u-\tau_{h}u$
and for $q=\theta p+(1-\theta)p^{\star}$ that 
\begin{align*}
\sumeps{x\in Q^{\eps}}\left|\left(u-\tau_{h}u\right)(x)\right|^{q} & \leq\left\Vert u-\tau_{h}u\right\Vert _{L^{p}(Q^{\eps})}^{p\theta}\left\Vert u-\tau_{h}u\right\Vert _{L^{p^{\star}}(Q^{\eps})}^{\left(1-\theta\right)p^{\star}}\,,
\end{align*}
and hence (\ref{eq:FKR-conditions-3}) follows from Theorem \ref{thm:discr-Poincare-wsp}
and (\ref{eq:FKR-conditions-3-help}).
\end{proof}
\bibliographystyle{plain}
\bibliography{RefHomogenRCM}

\end{document}